\documentclass{article}
\usepackage{gastex}
\usepackage{amsmath}
\usepackage{amssymb}
\usepackage{upref}
\usepackage{multirow}
\usepackage{multicol}
\usepackage{url}
\usepackage[all]{xy}
\usepackage{vaucanson-g}
\usepackage{color}
\usepackage{makeidx}
\makeindex
\usepackage{theorem}
\newtheorem{theorem}{Theorem}
\newtheorem{lemma}[theorem]{Lemma}
\newtheorem{proposition}[theorem]{Proposition}
\newtheorem{corollary}[theorem]{Corollary}
{\theorembodyfont{\rmfamily}%
  \newtheorem{example}[theorem]{Example}
   }
\newenvironment{proof}{\noindent\textit{Proof.}}
{\QED\vskip\theorempostskipamount} 
\newenvironment{proofof}[1]{\noindent\textit{Proof
    \protect{#1}.}}
                       {\QED\vskip\theorempostskipamount}
\def\petitcarre{\vrule height4pt width 4pt depth0pt}
\def\QED{\relax\ifmmode\eqno{\hbox{\petitcarre}}\else{%
  \unskip\nobreak\hfil\penalty50\hskip2em\hbox{}\nobreak\hfil
  \petitcarre
  \parfillskip=0pt \finalhyphendemerits=0\par\smallskip}
  \fi}
\newcommand\alp{\mathop{\rm alph}\nolimits}
\newcommand\A{\mathcal{A}}

\newcommand\RR{\mathcal{R}}

\newcommand{\N}{\mathbb{N}}
\newcommand{\Z}{\mathbb{Z}}
\newcommand{\R}{\mathbb{R}}

\let\edge\xrightarrow
\def\un(#1){\underline{#1}\,}
\DeclareMathOperator{\Card}{Card}

\DeclareMathOperator{\Sep}{Sep}

\definecolor{ivoire}{rgb}{0.99,0.99,0.8}
\DeclareMathOperator{\Fac}{Fac}
\def\bh{\mathbf{h}}
\def\id{{\rm id}}

\definecolor{light-gray}{gray}{0.7}
\usepackage{calc}
\newcounter{hours}\newcounter{minutes}
\newcommand\computetime{\setcounter{hours}{\time/60}%
  \setcounter{minutes}{\time-\value{hours}*60}%
  \thehours\,h\,\theminutes}
\newcommand\dateandtime{\today\quad\computetime}
\numberwithin{theorem}{section}
\numberwithin{equation}{section}
\numberwithin{figure}{section}
\numberwithin{table}{section}
\title{Maximal bifix decoding}
\author{Val\'erie Berth\'e$^1$, Clelia De Felice$^2$, 
Francesco Dolce$^3$, Julien Leroy$^4$,\\
 Dominique Perrin$^3$,
Christophe  Reutenauer$^5$,
Giuseppina Rindone$^3$\\\\
$^1$CNRS, Universit\'e Paris 7,
$^2$Universit\`a degli Studi di Salerno,\\
$^3$Universit\'e Paris Est, LIGM,
$^4$Universit\'e du Luxembourg,\\
 $^5$Universit\'e du Qu\'ebec \`a Montr\'eal}

\date{\dateandtime}
\begin{document}
\makeatletter
\def\@listI{%
  \leftmargin\leftmargini
  \setlength{\parsep}{0pt plus 1pt minus 1pt}
  \setlength{\topsep}{2pt plus 1pt minus 1pt}
  \setlength{\itemsep}{0pt}
}
\let\@listi\@listI
\@listi
\def\@listii {%
  \leftmargin\leftmarginii
  \labelwidth\leftmarginii
  \advance\labelwidth-\labelsep
  \setlength{\topsep}{0pt plus 1pt minus 1pt}
}
\def\@listiii{%
  \leftmargin\leftmarginiii
  \labelwidth\leftmarginiii
  \advance\labelwidth-\labelsep
  \setlength{\topsep}{0pt plus 1pt minus 1pt}
  \setlength{\parsep}{0pt} 
  \setlength{\partopsep}{1pt plus 0pt minus 1pt}
}
\makeatother

\maketitle

\begin{abstract}
We consider a class of sets of words which is
 a natural common generalization of Sturmian
sets and of interval exchange sets. This class
of sets consists of the uniformly recurrent tree sets,
where the tree sets are defined by a condition on
the possible extensions of bispecial factors.
We prove that this
class  is closed under maximal bifix decoding.
The proof uses the fact that the class is also closed
 under decoding with respect to return words.
\end{abstract}
\tableofcontents
\section{Introduction}
This paper studies the properties of
 a common generalization of Sturmian sets and
regular interval exchange
sets.
We first give some elements on the background of these two
families of sets.  

Sturmian words are infinite words over a binary alphabet that have
exactly $n+1$ factors of length $n$ for each $n\ge 0$.  Their origin
can be traced back to the astronomer J. Bernoulli III.  Their first
in-depth study is by Morse and Hedlund~\cite{MorseHedlund1940}. Many
combinatorial properties were described in the paper by Coven and
Hedlund~\cite{CovenHedlund1973}.  

 We understand here by
Sturmian 
words the generalization to arbitrary alphabets,
often called strict episturmian words or Arnoux-Rauzy words
(see  the survey~\cite{GlenJustin2009}), of the classical Sturmian
words
on two letters. A Sturmian set is the set of factors of one Sturmian word.
For more details, see~\cite{PytheasFogg2002,Lothaire2002}.

Sturmian words are closely related to the free group. This connection is
one of the main points of the series of 
papers~\cite{BerstelDeFelicePerrinReutenauerRindone2012,BertheDeFeliceDolceLeroyPerrinReutenauerRindone2013a,BertheDeFeliceDolceLeroyPerrinReutenauerRindone2013}
and the present one. A striking feature of this connection is the fact that
our results do not hold only for two-letter alphabets or for two
generators but for any number of letters and generators.

Interval exchange transformations 
were introduced by Oseledec~\cite{Oseledec1966}
following
an earlier idea of Arnold~\cite{Arnold1963}. These transformations form a generalization
of rotations of the circle.
The class of regular interval exchange transformations was introduced
by Keane~\cite{Keane1975} who showed that they are minimal in the
sense of topological dynamics. The set of factors of the natural codings of 
a regular interval exchange transformation is called an
interval exchange set.

Even though they have the same factor complexity (that is, the same number
of
factors of a given length), 
Sturmian words and codings of interval exchange transformations have a priori
very distinct combinatorial behaviours, whether for the type of behaviour of their special factors,
or for balance properties and deviations of Birkhoff sums (see 
\cite{CassaigneFerencziMessaoudi2008,Zorich1997}).

The class of tree sets, introduced in~\cite{BertheDeFeliceDolceLeroyPerrinReutenauerRindone2013a},
contains both the Sturmian sets and the regular interval exchange
sets. They are defined by  a condition on the possible extensions of bispecial
factors. 

In a paper with
part of the present list of authors
on bifix codes and Sturmian words
\cite{BerstelDeFelicePerrinReutenauerRindone2012} we proved
that Sturmian sets satisfy the finite index basis
property, in the sense that, given a set $S$ of words on an
alphabet $A$,  a finite bifix code
is $S$-maximal if and only if it is the basis of a subgroup
of finite index of the free group on $A$.
The main statement of
\cite{BertheDeFeliceDolceLeroyPerrinReutenauerRindone2013}
is that uniformly recurrent tree sets satisfy the finite index basis
property.
This generalizes the result concerning Sturmian words
of~\cite{BerstelDeFelicePerrinReutenauerRindone2012} quoted above. 
As
an
example of a consequence of this result, if $S$ is a uniformly
recurrent tree set on the alphabet $A$, then for any $n\ge 1$,
the set $S\cap A^n$ is a basis of the subgroup formed by the words of
length
multiple of $n$ (see Theorem~\ref{theoremGroupCode}).

 Our main result here is that
the class of uniformly recurrent tree sets is closed under maximal
bifix decoding (Theorem~\ref{theoremNormal}). This means that if $S$ is a uniformly recurrent tree set
 and $f$ a coding morphism for a finite
$S$-maximal bifix code, then $f^{-1}(S)$ is a uniformly recurrent tree set.
 The family of regular interval
exchange sets is closed under maximal bifix decoding 
(see~\cite[Theorem 3.13]{BertheDeFeliceDolceLeroyPerrinReutenauerRindone2013ie})
but the family of Sturmian sets is not (see Example~\ref{exampleTribonacci2}
below). Thus, this result shows that the family of uniformly recurrent tree
sets is the natural
closure of the family of Sturmian sets. 

The proof of Theorem~\ref{theoremNormal} uses the
finite index basis property of uniformly recurrent tree sets.
It also  uses the closure of uniformly
recurrent
tree sets under decoding with respect to return words 
(Theorem~\ref{propositionReturns}). This property,
which is interesting in its own, generalizes the fact that the derived
word of a Sturmian word is Sturmian~\cite{JustinVuillon2000}.

The paper is organized as follows. In
Section~\ref{sectionPreliminaries},
 we introduce 
the notation and recall some basic results.
We define the composition of  codes.

In Section~\ref{sectionIntervalExchange}, we introduce one important
subclass of tree sets, namely interval
exchange sets. We recall the definitions concerning minimal
and regular interval
exchange transformations.
 We prove in~\cite{BertheDeFeliceDolceLeroyPerrinReutenauerRindone2013ie} that the
class of regular interval exchange sets is closed under maximal
bifix decoding.

In Section~\ref{sectionReturn}, we define return words, derived words
and derived sets and prove some elementary properties.

In Section~\ref{sectionTreeNormal}, we recall the definition
of tree sets. We also
recall that a regular interval exchange set is a tree set
(Proposition~\ref{propositionExchangeTreeCondition}). 
 We prove that the family of uniformly recurrent tree sets
is closed under derivation (Theorem~\ref{propositionReturns}).
We
further
prove that all bases of the free group included in a uniformly
recurrent tree set are tame, that is, obtained from the alphabet by composition
of elementary positive automorphisms (Theorem~\ref{theoremTame}).

In Section~\ref{sectionSadic}, we turn to the notion of $S$-adic representation
of sets, introduced in~\cite{Ferenczi1996}, using a terminology
initiated by Vershik and coined out by B.~Host.
We  deduce from the previous result that uniformly recurrent tree sets have a 
primitive
$S_e$-adic representation (Theorem~\ref{base tame}) where $S_e$
is the finite set of positive elementary automorphisms of the free group.
In the case of a ternary alphabet, using results from~\cite{Leroy2014},
this result can be refined to
a characterization of the $S$-adic representation of 
tree sets~\cite{Leroy2014bis}.

In Section~\ref{sectionBifixDecoding}, we state and
prove our main result (Theorem~\ref{theoremNormal}), namely the
closure under maximal bifix decoding of the family of uniformly recurrent tree sets.

Finally, in Section~\ref{sectionComposition}, we use
Theorem~\ref{theoremNormal} to prove a result concerning the composition
of bifix codes (Theorem~\ref{theoremCompositionBifix})
showing that the degrees of the terms of a composition are multiplicative.

\paragraph{Acknowledgments} The authors wish to thank the referees for their
suggestions which helped to improve the presentation of the paper.

This work was supported by grants from
Region Ile-de-France, the ANR projects Dyna3S and Eqinocs, the FARB Project
``Aspetti algebrici e computazionali nella teoria dei codici,
degli automi e dei linguaggi formali'' (University of Salerno, 2013)
and the MIUR PRIN 2010-2011 grant
``Automata and Formal Languages: Mathematical and Applicative Aspects''.
\section{Preliminaries}\label{sectionPreliminaries}
In this section, we recall some notions and definitions concerning words,
 codes and automata. For a more detailed presentation, 
see~\cite{BerstelDeFelicePerrinReutenauerRindone2012}.
We also introduce the notion of composition of codes.
\subsection{Words}\label{subsectionWords}
Let $A$ be a finite nonempty alphabet. All words considered below,
unless
stated explicitly, are supposed to be on the alphabet $A$.
We let $A^*$ denote the set of all finite words over $A$ and  $A^+$ the set of finite nonempty words over $A$.
The empty word is denoted by $1$ or by $\varepsilon$ .
We let $|w|$ denote the length of a word $w$.
For a set $X$ of words and a word $x$, we denote
\begin{displaymath}
x^{-1}X=\{y\in A^*\mid xy\in X\},\quad Xx^{-1}=\{z\in A^*\mid zx\in X\}.
\end{displaymath}
A finite word $v$ is a \emph{factor} of a (possibly infinite)
word $x$ if $x=uvw$.
A set of words is said to be \emph{factorial} if it contains the
factors of its elements.
Let $S$ be a set of finite words on the alphabet $A$.
 For $w\in S$,
we denote
\begin{displaymath}
L(w)=\{a\in A\mid aw\in S\},\
R(w)=\{a\in A\mid wa\in S\},
\end{displaymath}
\begin{displaymath}
E(w)=\{(a,b)\in A\times A\mid awb\in S\}
\end{displaymath}
and further
\begin{displaymath}
\ell(w)=\Card(L(w)),\quad r(w)=\Card(R(w)),\quad e(w)=\Card(E(w)).
\end{displaymath}
These notions depend upon $S$ but it is assumed from the context.
A word $w$ is \emph{right-extendable} if $r(w)>0$,
\emph{left-extendable} if $\ell(w)>0$ and \emph{biextendable} if
$e(w)>0$. A 
factorial set
$S$ is called \emph{right-extendable}
(resp. \emph{left-extendable}, resp. \emph{biextendable}) if every word in $S$ is
right-extendable (resp. left-extendable, resp. biextendable).

A word $w$ is called
\emph{right-special}
if $r(w)\ge 2$. It is called \emph{left-special} if $\ell(w)\ge 2$.
It is called \emph{bispecial} if it is both right and left-special.

We let $\Fac(x)$ denote the set of factors of an infinite word $x\in
A^\N$.  The set $\Fac(x)$ is factorial and right-extendable.
An infinite word $x\in A^\omega$ is \emph{recurrent} if for every $u\in \Fac(x)$
there is a word $v$ such that $uvu\in \Fac(x)$.

A factorial set of words $S\ne\{1\}$ is \emph{recurrent}  if for every
$u,w\in S$ there is a word $v$ such that $uvw\in S$.
For any recurrent set $S$ there is an infinite word $x$ such that $\Fac(x)=S$
(see~\cite[Proposition 2.2.1]{BerstelDeFelicePerrinReutenauerRindone2012}). 

For every infinite word $x$, the set $\Fac(x)$ is recurrent if and only if $x$
is recurrent (see~\cite[Proposition 2.2.2]{BerstelDeFelicePerrinReutenauerRindone2012}).

A set of words $S$ is said to be \emph{uniformly recurrent} if it is
right-extendable and if, for any word $u\in S$, there exists an integer $n\ge
1$
such that $u$ is a factor of every word of $S$ of length $n$.
A uniformly recurrent set is recurrent.

A \emph{morphism} $f:A^*\rightarrow B^*$ is a monoid morphism from
$A^*$ to $B^*$. If $a\in A$ is such that the word $f(a)$ begins with
$a$ and if $|f^n(a)|$ tends to infinity with $n$, there is a unique
infinite word denoted $f^\omega(a)$ which has all words $f^n(a)$
as prefixes. It is called a \emph{fixed point} of the morphism $f$.

A morphism $f:A^*\rightarrow A^*$ is called \emph{primitive} if there
is an integer $k$ such that for all $a,b\in A$, the letter $b$
appears in $f^k(a)$. If $f$ is a primitive morphism, the set
of factors of any fixed point
of $f$ is uniformly recurrent (see~\cite[Proposition 1.2.3]{PytheasFogg2002}
 for example).

An infinite word is \emph{episturmian} if the set of its factors 
is closed under reversal and contains for each $n$ at most one word
of length $n$ which is right-special. It is a \emph{strict episturmian} word
if it has exactly one right-special word of each length and
moreover each right-special factor $u$ is such that $r(u)=\Card(A)$.


A \emph{Sturmian set} is a set of words which 
is the set of factors of a strict episturmian word.
Any Sturmian set is uniformly recurrent (see
\cite[Proposition 2.3.3]{BerstelDeFelicePerrinReutenauerRindone2012} for example).
\begin{example}\label{exampleFibonacci}
Let $A=\{a,b\}$.
The Fibonacci word is the fixed point $x=abaababa\ldots$ of the
morphism $f:A^*\rightarrow A^*$ defined by $f(a)=ab$ and $f(b)=a$.
It is a Sturmian word (see~\cite{Lothaire2002}). The set $\Fac(x)$ of factors of $x$ is the
\emph{Fibonacci set}.
\end{example}
\begin{example}\label{exampleTribonacci}
Let $A=\{a,b,c\}$.
The Tribonacci word is the fixed point $x=f^\omega(a)=abacaba\cdots$ of the morphism
$f:A^*\rightarrow A^*$ defined by $f(a)=ab$, $f(b)=ac$, $f(c)=a$.
It is a strict episturmian word (see~\cite{JustinVuillon2000}).
The set $\Fac(x)$ of factors of $x$ is the \emph{Tribonacci set}.
\end{example}

\subsection{Bifix codes}
Recall that a set $X\subset A^+$ of nonempty words over an alphabet
$A$ is a \emph{code}\index{code} if the relation
\begin{displaymath}
  x_1\cdots x_n=y_1\cdots y_m
\end{displaymath}
with $n,m\ge1$ and $x_1,\ldots,x_n,y_1,\ldots,y_m\in X$ implies $n=m$
and $x_i=y_i$ for $i=1,\ldots, n$. For the general theory of codes,
see~\cite{BerstelPerrinReutenauer2009}.

A \emph{prefix code} is a set of nonempty words which does not contain any
proper prefix of its elements. A prefix code is a code.

A suffix code is defined symmetrically.
A  \emph{bifix code} is a set which is both a prefix code and a suffix
code.

A \emph{coding morphism} for a  code $X\subset A^+$ is a morphism
$f:B^*\rightarrow A^*$ which maps bijectively $B$ onto $X$.

Let $S$ be a  set of words. A prefix code $X\subset S$ is $S$-maximal 
if it is not properly contained in any prefix code 
$Y\subset S$. Equivalently, a prefix code
$X\subset S$ is $S$-maximal if every word in $S$ is comparable
for the prefix order with some word of $X$.

A set of words $M$ is called \emph{right unitary}
 if $u,uv\in M$ imply $v\in M$.
The submonoid $M$ generated by a prefix code 
is right unitary.
One can show that conversely, any right unitary submonoid of $A^*$
is generated by a prefix code
(see~\cite{BerstelPerrinReutenauer2009}).
The symmetric notion of a \emph{left unitary} set is defined
by the condition $v,uv\in M$ implies $u\in M$.

We denote by $X^*$ the submonoid generated by $X$.
A set $X\subset S$ is \emph{right $S$-complete} if every word of $S$
is a prefix of a word in $X^*$. If $S$ is factorial,
a prefix code is $S$-maximal if
and only if it is right $S$-complete \cite[Proposition 3.3.2]{BerstelDeFelicePerrinReutenauerRindone2012}.

Similarly a bifix code $X\subset S$ is $S$-maximal if it is
not properly contained in a bifix code $Y\subset S$.
For a recurrent set $S$, a finite bifix code is $S$-maximal as a bifix code if
and only if it is an $S$-maximal prefix code 
\cite[Theorem
4.2.2]{BerstelDeFelicePerrinReutenauerRindone2012}. For a uniformly recurrent set $S$, any finite bifix code 
$X\subset S$ is contained in a finite $S$-maximal bifix code
\cite[Theorem 4.4.3]{BerstelDeFelicePerrinReutenauerRindone2012}.

A \emph{parse} of a word $w\in A^*$ with respect to a set $X$ is
a triple $(v,x,u)$ such that $w=vxu$ where $v$ has no suffix in $X$,
$u$ has no prefix in $X$ and $x\in X^*$.
We denote by $d_X(w)$ the number of parses of $w$ with respect to $X$.

Let $X$ be a bifix code.
The number of parses of a word $w$ is also equal to the number
of suffixes of $w$ which have no prefix in $X$ and the
 number of prefixes of $w$ which have no suffix in $X$
\cite[Proposition 6.1.6]{BerstelPerrinReutenauer2009}.

By definition, the $S$-degree of a bifix code
$X$, denoted $d_X(S)$, is the maximal number
of parses of all words in $S$ with respect to $X$.  It can be finite or infinite.

The set of \emph{internal factors} of a set of words $X$,
denoted $I(X)$,
is the set of words $w$ such that 
there exist nonempty words $u,v$ with $uwv\in X$.

Let $S$ be a recurrent set and let
 $X$ be a finite $S$-maximal bifix code of $S$-degree $d$.
A word $w\in S$ is such that
$d_X(w)< d$ if and only if it is an internal factor of $X$,
that is
\begin{equation}
I(X)=\{w\in S\mid d_X(w)<d\}\label{eqInternal}
\end{equation}
\cite[Theorem 4.2.8]{BerstelDeFelicePerrinReutenauerRindone2012}.
Thus any word of $X$ of maximal length has $d$ parses.
This implies that
the $S$-degree $d$ is finite.
\begin{example}
Let $S$ be a recurrent set. For any integer $n\ge 1$, the set
$S\cap A^n$ is an $S$-maximal bifix code of $S$-degree $n$.
\end{example}

The \emph{kernel}\index{kernel} of a set of words $X$ is the set of
words in $X$ which are internal factors of words in $X$. We let
$K(X)$  denote the kernel of $X$. Note that $K(X)=I(X)\cap X$.

For any recurrent set $S$, a finite $S$-maximal bifix code is
determined by its $S$-degree and its kernel 
(see~\cite[Theorem 4.3.11]{BerstelDeFelicePerrinReutenauerRindone2012}).
\begin{example}\label{exampleDegree1}
Let $S$ be a recurrent set containing the alphabet $A$. 
The only $S$-maximal bifix code of
$S$-degree
$1$ is the alphabet $A$. This is clear since $A$ is the unique
$S$-maximal bifix code of $S$-degree $1$ with empty kernel.
\end{example}
\subsection{Group codes}\label{subsectionautomata}
We let $\A=(Q,i,T)$ denote a deterministic automaton with $Q$ as set of
states,
$i\in Q$ as initial state and $T\subset Q$ as set of terminal states.
For $p\in Q$ and $w\in A^*$, we denote $p\cdot w=q$ if
 there is a path labeled $w$ from $p$ to the state $q$ and $p\cdot
 w=\emptyset$
otherwise (for a general introduction to automata theory, see~\cite{Eilenberg1974} or~\cite{Sakarovitch2009}, for example).

The set \emph{recognized}\index{recognized by an automaton} by the
automaton is the set of words $w\in A^*$ such that $i\cdot w\in T$. A
set of words is \emph{rational}\index{rational set} if it is recognized
by a finite automaton. Two automata are \emph{equivalent} if they
recognize the same set.

All automata considered in this paper are deterministic and we simply
call them  `automata' to mean `deterministic automata'.

The automaton $\A$ is \emph{trim} if for every $q\in Q$, there is a path from $i$ to $q$ and
a path from $q$ to some $t\in T$.

An automaton is called
\emph{simple}\index{automaton!simple}\index{simple automaton} if it is
trim and if it has a unique terminal state which coincides with the
initial state.

An automaton $\A=(Q,i,T)$ is \emph{complete}%
\index{automaton!complete}\index{complete automaton}
if for every state $p\in Q$
and every letter $a\in A$, one has $p\cdot a\ne\emptyset$.

For a nonempty set $L\subset A^*$, we denote by $\A(L)$ the \emph{minimal
  automaton}\index{automaton!minimal}\index{minimal automaton}
 of $L$. The states of $\A(L)$ are the nonempty sets
$u^{-1}L=\{v\in A^*\mid uv\in L\}$ for $u\in A^*$
(see Section~\ref{subsectionWords} for the notation $u^{-1}L$).  
For $u\in A^*$ and $a\in A$, one defines $(u^{-1}L)\cdot a=(ua)^{-1}L$.
The initial state
is the set $L$ and the terminal states are the sets $u^{-1}L$ for
$u\in L$.


Let $X\subset A^*$ be a prefix code. Then there is a simple automaton
$\A=(Q,1,1)$ that recognizes $X^*$. Moreover, the minimal
automaton of $X^*$ is simple.
\begin{example}
The automaton $\A=(Q,1,1)$ represented in Figure~\ref{figureExampleAutomaton} is the minimal
automaton of $X^*$ with $X=\{aa,ab,ac,ba,ca\}$.
\begin{figure}[hbt]
\centering\gasset{Nadjust=wh}
\begin{picture}(60,10)
\node(3)(0,5){$3$}\node[Nmarks=if,iangle=90,fangle=90](1)(20,5){$1$}
\node(2)(40,5){$2$}

\drawedge[curvedepth=3](1,2){$a$}\drawedge[curvedepth=3](2,1){$a,b,c$}
\drawedge[curvedepth=3](1,3){$b,c$}\drawedge[curvedepth=3](3,1){$a$}

\end{picture}
\caption{The minimal automaton of $\{aa,ab,ac,ba,ca\}^*$.}
\label{figureExampleAutomaton}
\end{figure}
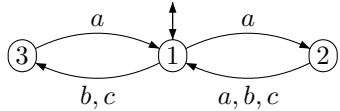
We have $Q=\{1,2,3\}$, $i=1$ and $T=\{1\}$. The initial state is indicated
by
an incoming arrow and the terminal one by an outgoing arrow.
\end{example}

An automaton $\A=(Q,1,1)$
is a \emph{group automaton}\index{group!automaton}\index{automaton!group}
if for every letter $a\in A$ the map $\varphi_\A(a):p\mapsto p\cdot a$
is a permutation of $Q$.

The following result is proved in~\cite[Proposition 6.1.5]{BerstelDeFelicePerrinReutenauerRindone2012}.
\begin{proposition}\label{propositionGroupAutomaton}
  The following conditions are equivalent for a submonoid $M$ of
  $A^*$.
  \begin{enumerate}
  \item[\rm (i)] $M$ is recognized by a group automaton with $d$
    states.
  \item[\rm (ii)] $M=\varphi^{-1}(K)$, where $K$ is a subgroup of
    index $d$ of a group $G$ and $\varphi$ is a surjective morphism
    from $A^*$ onto $G$.
  \item[\rm (iii)] $M=H\cap A^*$, where $H$ is a subgroup of index $d$
    of the free group on $A$.
  \end{enumerate}
  If one of these conditions holds, the minimal generating set of $M$
  is a maximal bifix code of degree $d$.
\end{proposition}

A bifix code $Z$ such that $Z^*$ satisfies one of the equivalent
conditions of Proposition~\ref{propositionGroupAutomaton} is
called a \emph{group code}\index{group!code}\index{code!group}
of degree $d$.

\subsection{Composition of codes}
We introduce the notion of composition of codes 
(see~\cite{BerstelPerrinReutenauer2009} for a more detailed
presentation).

For a set $X\subset A^*$, we denote by $\alp(X)$ the set of letters 
$a\in A$ which appear in the words of $X$.

Let $Z \subset A^*$ and $Y \subset B^*$ be two finite codes with $B =
\alp(Y)$. Then the codes $Y$ and $Z$ are {\em
  composable} if there is a bijection from $B$
onto $Z$. Since $Z$ is a code, this bijection defines  an injective
 morphism  from $B^*$ into $A^*$.
 If $f$ is such a morphism, then $Y$ and $Z$ are called
composable {\em through} $f$.  The set
\begin{equation}\label{eq1.6.1}
  X=f(Y)\subset Z^*\subset A^*           
\end{equation}
is obtained by {\em composition}
 of $Y$ and
$Z$ (by means of $f$). We denote it by
$X=Y\circ_f Z$, or by $X=Y\circ Z$ when the context permits it. Since $f$ is
injective, $X$ and $Y$ are related by bijection, and in particular
$\Card(X) = \Card(Y)$.  The words in $X$ are obtained just by
replacing, in the words of $Y$, each letter $b$ by the word $f(b)
\in Z$.

\begin{example}
Let $A=\{a,b\}$ and $B=\{u,v,w\}$. Let $f:B^*\rightarrow A^*$ be the
morphism defined by $f(u)=aa$, $f(v)=ab$ and $f(w)=ba$. Let
$Y=\{u,vu,vv,w\}$
and $Z=\{aa,ab,ba\}$. Then $Y,Z$ are composable through $f$ and
$Y\circ_f Z=\{aa,abaa,abab,ba\}$.
\end{example}
  If $Y$ and $Z$ are two composable codes, then $X = Y \circ Z$ is a
code \cite[Proposition 2.6.1]{BerstelPerrinReutenauer2009}
and if $Y$ and $Z$ are prefix (suffix) codes, then
$X$ is a prefix (suffix) code.
Conversely, if $X$ is a prefix (suffix) code, then $Y$ is a prefix
(suffix) code. 

We extend the notation $\alp$ as follows. For two codes $X,Z\subset
A^*$
we denote $\alp_Z(X)$ the set of $z\in Z$ such that $uzv\in X$
for some $u,v\in Z^*$.
The following is Proposition 2.6.6
in~\cite{BerstelPerrinReutenauer2009}.

\begin{proposition}\label{prop266}
Let $X,Z\subset A^*$ be codes. There exists a code $Y$ such that
$X=Y\circ Z$ if and only if $X\subset Z^*$ and $\alp_Z(X)=Z$.
\end{proposition}

The following statement generalizes Propositions 2.6.4 and 2.6.12
of~\cite{BerstelPerrinReutenauer2009} for prefix codes.
\begin{proposition}\label{propositionMaxPref}
Let $Y,Z$ be   finite prefix codes
 composable through $f$  and let $X=Y\circ_f Z$.
\begin{enumerate}
\item[\rm(i)] For every set $T$ such that $Y\subset T$ and $Y$
is a $T$-maximal prefix
  code, $X$ is an $f(T)$-maximal prefix code.
\item[\rm(ii)] For every set $S$  such that $X,Z\subset S$, if $X$ is an
$S$-maximal prefix code,  $Y$ is an
  $f^{-1}(S)$-maximal
prefix code and $Z$ is an $S$-maximal prefix code.
The converse is true if $S$ is recurrent.
\end{enumerate}
\end{proposition}
\begin{proof}
\noindent(i) Let $w\in f(T)$ and set $w=f(v)$ with $v\in T$.
 Since $Y$ is $T$-maximal, there
is a word $y\in Y$ which is prefix-comparable with
 $v$. Then
$f(y)$ is  prefix-comparable with $w$. Thus $X$ is $f(T)$-maximal.

\noindent(ii) Since $X$ is an $S$-maximal prefix code,  any word
in $S$ is prefix-comparable with some element of $X$ and thus
with some element of $Z$. Therefore, $Z$ is $S$-maximal. Next
if $u\in f^{-1}(S)$, $v=f(u)$ is in $S$ and is prefix-comparable with a word $x$
in $X$. Assume that $v=xt$. Then $t$ is in $Z^*$ since $v,x\in Z^*$. Set $w=f^{-1}(t)$
and $y=f^{-1}(x)$.
Since $u=yw$, $u$ is prefix-comparable with $y$ which is
in $Y$. The other case is similar.

Conversely, assume that $S$ is
recurrent. Let $w$ be a word in $S$ of length strictly larger than the sum
of the maximal length of the words of $X$ and $Z$.
Since $S$ is recurrent, the set $Z$ is right $S$-complete, 
and consequently the word $w$ is a
prefix
of a word in $Z^*$. Thus $w=up$ with $u\in Z^*$ and $p$ a proper prefix of a
word
in $Z$. The hypothesis on $w$ implies that $u$ is longer than any word
of $X$. Let $v=f^{-1}(u)$. Since $u\in S$, we have $v\in f^{-1}(S)$. It is
not possible that $v$ is a proper prefix of a word of $Y$ since
otherwise $u$ would be shorter than a word of $X$. Thus $v$ has
a prefix in $Y$. Consequently $u$, and thus $w$,  has a prefix in $X$.
Thus $X$ is $S$-maximal.
\end{proof}

Note that the converse of (ii) is not true if the hypothesis that $S$ is
recurrent
is replaced by factorial. Indeed, for $S=\{1,a,b,aa,ab,ba\}$, $Z=\{a,ba\}$,
 $Y=\{uu,v\}$, $f(u)=a$ and $f(v)=ba$,
one has $f^{-1}(S)=\{1,u,uu,v\}$ and $X=\{aa,ba\}$, which is
not an $S$-maximal prefix code.

Note also that when $S$ is recurrent (or even uniformly recurrent), the set
$T=f^{-1}(S)$ need not be recurrent. Indeed, let $S$ be the set
of factors of $(ab)^*$, let $B=\{u,v\}$ and let $f:B^*\rightarrow A^*$
be defined by $f(u)=ab$, $f(v)=ba$. Then $T=u^*\cup v^*$ which is not
recurrent.

\section{Interval exchange sets}\label{sectionIntervalExchange}
In this section, we recall the definition and the basic properties
of interval exchange transformations. 
\subsection{Interval exchange transformations}
Let us recall the definition of an interval exchange transformation
(see~\cite{CornfeldFominSinai1982} or~\cite{BertheRigo2010}).

A \emph{semi-interval} is a nonempty subset of the real line of the
form $[\alpha,\beta)=\{z\in\R\mid \alpha\le z<\beta\}$. Thus it
is a left-closed and right-open interval. For two semi-intervals
$\Delta,\Gamma$, we denote $\Delta<\Gamma$ if $x<y$
for any $x\in\Delta$ and $y\in\Gamma$. 

Let $(A,<)$ be an ordered set. A partition $(I_a)_{a\in A}$
of $[0,1)$ in semi-intervals is \emph{ordered}  if $a<b$
implies $I_a<I_b$.

Let $A$ be a finite set ordered by two total orders $<_1$ and $<_2$. Let 
$(I_a)_{a\in A}$ be  a partition of $[0,1)$ in semi-intervals ordered
  for $<_1$.
Let $\lambda_a$ be the length of $I_a$.
Let $\mu_a=\sum_{b\le_1 a}\lambda_b$ and $\nu_a=\sum_{b\le_2
  a}\lambda_b$.
Set $\alpha_a=\nu_a-\mu_a$.
The \emph{interval exchange transformation} relative to 
$(I_a)_{a\in A}$ is the  map $T:[0,1)\rightarrow [0,1)$ defined by
\begin{displaymath}
 T(z)=z+\alpha_a\quad \text{ if } z\in I_a.
\end{displaymath}
Observe that the restriction of $T$ to $I_a$ is a translation
onto $J_a=T(I_a)$, that $\mu_a$ is the right boundary of $I_a$
and that $\nu_a$ is the right boundary of $J_a$.
We additionally denote
 by $\gamma_a$ the left boundary of $I_a$ and
by $\delta_a$ the left boundary of $J_a$. Thus
$I_a=[\gamma_a,\mu_a)$, $J_a=[\delta_a,\nu_a)$.

Since $a<_2 b$ implies $J_a<_2J_b$, the family
$(J_a)_{a\in A}$ is a partition of $[0,1)$ ordered for $<_2$.
In particular,
the transformation $T$ defines a bijection from $[0,1)$ onto itself.

An interval exchange transformation relative to 
$(I_a)_{a\in A}$ is also said to be on the alphabet $A$.
The values $(\alpha_a)_{a\in A}$ are called the \emph{translation
  values}
of the transformation $T$.
\begin{example}\label{exampleRotation}
 Let $R$ be the interval
exchange
transformation corresponding to $A=\{a,b\}$, $a<_1b$, $b<_2a$,
$I_a=[0,1-\alpha)$,
$I_b=[1-\alpha,1)$ with $0<\alpha<1$.
The transformation $R$ is the rotation  of angle $\alpha$ on the semi-interval $[0,1)$ 
defined by $R(z)=z+\alpha\bmod 1$.
\end{example}
Since $<_1$ and $<_2$ are total
orders, there exists a unique permutation $\pi$ of $A$
such that $a<_1b$ if and only if $\pi(a)<_2\pi(b)$.
Conversely, $<_2$ is determined by $<_1$ and $\pi$,
and $<_1$ is determined by $<_2$ and $\pi$.
The permutation $\pi$ is said to be \emph{associated} with $T$.

Let $s\ge 2$ be an integer. If we set $A=\{a_1,a_2,\ldots,a_s\}$ with $a_1<_1a_2<_1\cdots<_1a_s$,
  the pair $(\lambda,\pi)$ formed by the family
$\lambda=(\lambda_a)_{a\in A}$ and the permutation
$\pi$
determines the map $T$. We will also denote $T$ as $T_{\lambda,\pi}$.
The transformation $T$ is also said to be an $s$-interval exchange
transformation.

It is easy to verify that the family of $s$-interval exchange
transformations
is closed by composition and by taking inverses.

\begin{example}
A $3$-interval exchange transformation is represented in
Figure~\ref{figure3interval}. One has $A=\{a,b,c\}$ with
$a<_1b<_1c$ and $b<_2c<_2a$.
The associated permutation is the cycle $\pi=(abc)$.
\begin{figure}[hbt]
\centering
\centering\gasset{Nh=2,Nw=2,ExtNL=y,NLdist=2,AHnb=0,ELside=r}

\begin{picture}(60,20)(0,-5)

\node[fillcolor=red](0H)(0,15){}
\node[fillcolor=blue](1H)(32.4,15){$\mu_a$}
\node[fillcolor=green](2H)(49.8,15){$\mu_b$}
\node(3H)(60,15){$\mu_c$}
\node[Nframe=n](05H)(16,15){}\node[Nframe=n](15H)(40,15){}\node[Nframe=n](25H)(55,15){}

\node[fillcolor=blue](0B)(0,0){}
\node[fillcolor=green](1B)(17.4,0){$\nu_b$}
\node[fillcolor=red](2B)(27.6,0){$\nu_c$}
\node(3B)(60,0){$\nu_a$}
\node[Nframe=n](05B)(5,0){}\node[Nframe=n](15B)(22,0){}\node[Nframe=n](25B)(43,0){}

\drawedge[linecolor=red,linewidth=1](0H,1H){}
\drawedge[linecolor=blue,linewidth=1](1H,2H){}
\drawedge[linecolor=green,linewidth=1](2H,3H){}
\drawedge[linecolor=blue,linewidth=1](0B,1B){}\drawedge[linecolor=green,linewidth=1](1B,2B){}
\drawedge[linecolor=red,linewidth=1](2B,3B){}
\drawedge[AHnb=1](05H,25B){}\drawedge[AHnb=1](15H,05B){}
\drawedge[AHnb=1](25H,15B){}
\end{picture}
\caption{A $3$-interval exchange transformation.}\label{figure3interval}
\end{figure}
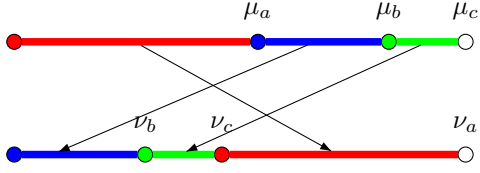
\end{example}

\subsection{Regular interval exchange transformations}

The \emph{orbit} of a point $z\in[0,1)$ is the set $\{T^n(z)\mid
  n\in\Z \}$. 
The transformation $T$ is said to be \emph{minimal} if for any
  $z\in[0,1)$,
the  orbit of $z$ is dense in $[0,1)$.

Set $A=\{a_1,a_2,\ldots,a_s\}$ with $a_1<_1 a_2<_1\ldots<_1 a_s$,
 $\mu_i=\mu_{a_i}$ and $\delta_i=\delta_{a_i}$. The points $0,\mu_1,\ldots,\mu_{s-1}$
form the set of \emph{separation points} of $T$, denoted $\Sep(T)$.

An interval exchange transformation $T_{\lambda,\pi}$ is called
\emph{regular} if the orbits of the nonzero separation
points $\mu_1,\ldots,\mu_{s-1}$ are infinite and disjoint.
Note that the orbit of $0$ cannot be disjoint of the others
since one has $T(\mu_i)=0$ for some $i$ with $1\le i\le s$.

\begin{example} The $2$-interval exchange transformation $R$
of Example~\ref{exampleRotation}
which is the rotation of angle $\alpha$ is regular
if and only if $\alpha$ is irrational.
\end{example}
The following result is due to Keane~\cite{Keane1975}.
\begin{theorem}\label{theoremKeane}
A regular interval exchange transformation is minimal.
\end{theorem}
Note that the converse is not true (see~\cite{BertheDeFeliceDolceLeroyPerrinReutenauerRindone2013ie} for an example).

\subsection{Natural coding}
 Let $T$ be an interval exchange transformation relative to 
$(I_a)_{a\in A}$. 
For a given real number $z\in[0,1)$,
the \emph{natural coding} of $T$ relative to $z$ is the infinite word
$\Sigma_T(z)=a_0a_1\cdots$ on the alphabet $A$
defined by 
\begin{displaymath}
a_n=a\quad \text{ if }\quad T^n(z)\in I_{a}.
\end{displaymath}

\begin{example}\label{exampleFiboNatCoding}
Let $\alpha=(3-\sqrt{5})/2$ and let $R$ be the rotation of
angle $\alpha$ on $[0,1)$ as in Example~\ref{exampleRotation}.
The natural coding of $R$ with respect to $\alpha$ is the Fibonacci
word (see~\cite[Chapter 2]{Lothaire2002} for example).
\end{example}
For a word $w=b_0b_1\cdots b_{m-1}$, 
let $I_w$ be the set
\begin{equation}
I_w=I_{b_0}\cap T^{-1}(I_{b_1})\cap\cdots\cap T^{-m+1}(I_{b_{m-1}}).\label{eqIu}
\end{equation}
Note that each $I_w$ is a semi-interval. Indeed, this is true if $w$ is
a letter. Next,  assume that $I_w$ is a semi-interval. Then
for any $a\in A$,
$T(I_{aw})=T(I_a)\cap I_w$ is a semi-interval since $T(I_a)$ is a
semi-interval
by definition of an interval exchange transformation. Since $I_{aw}\subset
I_a$, $T(I_{aw})$ is a translate of $I_{aw}$, which is therefore also
a semi-interval. This proves the property by induction on the length.

Then one has for any $n\ge 0$
\begin{equation}
a_na_{n+1}\cdots a_{n+m-1}=w \Longleftrightarrow T^n(z)\in I_w.\label{eqIw}
\end{equation}

If $T$ is minimal, one has $w\in \Fac(\Sigma_T(z))$ if
and only
if $I_w\ne\emptyset$.
Thus the set $\Fac(\Sigma_T(z))$ does not depend on $z$ (as for Sturmian words,
see~\cite{Lothaire2002}). Since it depends only on $T$, we denote
it by $\Fac(T)$. When $T$ is regular (resp. minimal), such a set is called a
\emph{regular interval exchange set} (resp. a minimal interval exchange set).

 The following
statement is well known (see~\cite{BertheDeFeliceDolceLeroyPerrinReutenauerRindone2013ie}).
\begin{proposition}\label{propositionRegularUR}
For any minimal interval exchange transformation $T$,
the set $\Fac(T)$ is uniformly recurrent.
\end{proposition}

\begin{example}\label{exampleDivision}
Set $\alpha=(3-\sqrt{5})/2$ and $A=\{a,b,c\}$.
Let $T$ be the interval exchange transformation on $[0,1)$ which is the
rotation of angle $2\alpha \bmod 1$ on the three intervals
$I_a=[0,1-2\alpha)$, $I_b=[1-2\alpha,1-\alpha)$, $I_c=[1-\alpha,1)$
(see Figure~\ref{figure3interval2}).
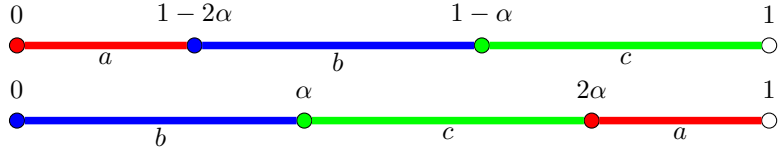
\begin{figure}[hbt]
\centering\gasset{Nh=2,Nw=2,ExtNL=y,NLdist=2,AHnb=0,ELside=r}
\begin{picture}(100,15)
\node[fillcolor=red](0h)(0,10){$0$}
\node[fillcolor=blue](1-2alpha)(23.6,10){$1-2\alpha$}
\node[fillcolor=green](1-alpha)(61.8,10){$1-\alpha$}
\node(1h)(100,10){$1$}
\drawedge[linecolor=red,linewidth=1](0h,1-2alpha){$a$}
\drawedge[linecolor=blue,linewidth=1](1-2alpha,1-alpha){$b$}
\drawedge[linecolor=green,linewidth=1](1-alpha,1h){$c$}

\node[fillcolor=blue](0b)(0,0){$0$}
\node[fillcolor=green](alpha)(38.2,0){$\alpha$}
\node[fillcolor=red](2alpha)(76.4,0){$2\alpha$}\node(1b)(100,0){$1$}
\drawedge[linecolor=blue,linewidth=1](0b,alpha){$b$}
\drawedge[linecolor=green,linewidth=1](alpha,2alpha){$c$}
\drawedge[linecolor=red,linewidth=1](2alpha,1b){$a$}
\end{picture}
\caption{A regular $3$-interval exchange transformation.}\label{figure3interval2}
\end{figure}
The transformation $T$ is regular since $\alpha$ is irrational. The
words of length at most $5$ of the set $S=\Fac(T)$ are represented in Figure~\ref{figureSetF} on the left.
\begin{figure}[hbt]
\centering
\gasset{Nadjust=wh,AHnb=0,ELpos=60,ELdist=.6}
\begin{picture}(120,50)
\put(0,0){
\begin{picture}(50,50)(0,5)
\node(1)(0,30){}\node(a)(10,40){}\node(b)(10,30){}\node(c)(10,20){}

\node(ac)(20,45){}
\node(ba)(20,35){}\node(bb)(20,25){}
\node(cb)(20,20){}\node(cc)(20,10){}

\node(acb)(30,50){}\node(acc)(30,45){}
\node(bac)(30,35){}\node(bba)(30,25){}
\node(cba)(30,20){}\node(cbb)(30,15){}
\node(ccb)(30,10){}

\node(acbb)(40,55){}\node(accb)(40,45){}
\node(bacb)(40,40){}\node(bacc)(40,35){}
\node(bbac)(40,25){}
\node(cbac)(40,20){}\node(cbba)(40,15){}
\node(ccba)(40,10){}\node(ccbb)(40,5){}

\node(acbba)(50,55){}\node(accba)(50,50){}
\node(accbb)(50,45){}
\node(bacbb)(50,40){}\node(baccb)(50,35){}
\node(bbacb)(50,30){}\node(bbacc)(50,25){}
\node(cbacc)(50,20){}
\node(cbbac)(50,15){}
\node(ccbac)(50,10){}\node(ccbba)(50,5){}


\drawedge(1,a){$a$}\drawedge(1,b){$b$}\drawedge(1,c){$c$}
\drawedge(a,ac){$c$}\drawedge(b,ba){$a$}\drawedge(b,bb){$b$}
\drawedge(c,cb){$b$}\drawedge(c,cc){$c$}
\drawedge(ac,acb){$b$}\drawedge(ac,acc){$c$}
\drawedge(ba,bac){$c$}\drawedge(bb,bba){$a$}
\drawedge(cb,cba){$a$}\drawedge(cb,cbb){$b$}
\drawedge(cc,ccb){$b$}
\drawedge(acb,acbb){$b$}\drawedge(acc,accb){$b$}
\drawedge(bac,bacb){$b$}\drawedge(bac,bacc){$c$}
\drawedge(bba,bbac){$c$}
\drawedge(cba,cbac){$c$}
\drawedge(cbb,cbba){$a$}\drawedge(cbba,cbbac){$c$}

\drawedge(acbb,acbba){$a$}\drawedge(accb,accba){$a$}
\drawedge(accb,accbb){$b$}\drawedge(bacb,bacbb){$b$}
\drawedge(bacc,baccb){$b$}
\drawedge(bbac,bbacb){$b$}\drawedge(bbac,bbacc){$c$}
\drawedge(cbac,cbacc){$c$}
\drawedge(cbba,cbbac){$c$}
\drawedge(ccb,ccba){$a$}\drawedge(ccb,ccbb){$b$}

\drawedge(ccba,ccbac){$c$}\drawedge(ccbb,ccbba){$a$}
\end{picture}
}
\put(60,0){
\begin{picture}(50,35)(0,5)
\node(1)(0,20){}

\node(a)(10,30){}\node(b)(10,20){}\node(c)(10,10){}

\node(ac)(20,35){}
\node(bb)(20,25){}\node(bc)(20,20){}
\node(ca)(20,10){}\node(cb)(20,5){}

\node(aca)(30,40){}\node(acb)(30,35){}
\node(bbb)(30,30){}\node(bbc)(30,25){}
\node(bcc)(30,20){}
\node(cac)(30,10){}
\node(cbb)(30,5){}

\drawedge(1,a){$a$}\drawedge[ELpos=60](1,b){$b$}\drawedge[ELpos=60](1,c){$c$}

\drawedge(a,ac){$c$}
\drawedge(b,bb){$b$}\drawedge[ELpos=60](b,bc){$c$}
\drawedge(c,ca){$a$}\drawedge[ELpos=60](c,cb){$b$}

\drawedge(ac,aca){$a$}\drawedge[ELpos=60](ac,acb){$b$}
\drawedge(bb,bbb){$b$}\drawedge(bb,bbc){$c$}
\drawedge(bc,bcc){$c$}
\drawedge(ca,cac){$c$}
\drawedge(cb,cbb){$b$}
\end{picture}
}
\end{picture}
\caption{The words of length $\le 5$ of the set $S$ and 
the words of length $\le 3$ of its derived set.}\label{figureSetF}
\end{figure}
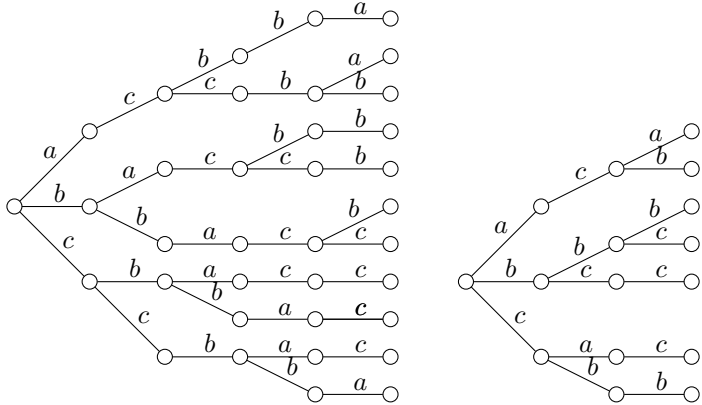
Since $T=R^2$, where $R$ is the transformation of Example~\ref{exampleFiboNatCoding},
the natural coding of $T$ relative to $\alpha$ is
the infinite word $y=\gamma^{-1}(x)$ where $x$ is the Fibonacci word
and $\gamma$ is the morphism defined by $\gamma(a)=aa$, $\gamma(b)=ab$,
$\gamma(c)=ba$. One has
\begin{equation}
y=baccbaccbbacbbacbbacc\cdots \label{Eqy}
\end{equation}
Actually, the word $y$ is the fixed point $g^\omega(b)$ of the morphism
$g:a\mapsto baccb, b\mapsto bacc, c\mapsto bacb$. This  follows from
the fact that the cube of the Fibonacci morphism $f:a\mapsto ab, b\mapsto a$ sends each letter
on a word of odd length and thus sends words of even
length on words of even length.
\end{example}

\section{Return words}\label{sectionReturn}

In this section, we introduce the notion of return and first
return words. We prove elementary results about return words
which essentially already appear in~\cite{Durand1998}.

Let $S$ be a  set of words. For $w\in S$, let
$\Gamma_S(w)=\{x\in S\mid wx\in S\cap A^+w\}$
be the set of \emph{right return words} to $w$ and let
$\RR_S(w)=\Gamma_S(w)\setminus\Gamma_S(w) A^+$ be the 
set of \emph{first right  return words} to $w$.
By definition, the set $\RR_S(w)$ is, for every $w\in S$, a prefix code.
If $S$ is recurrent, it is 
 a $w^{-1}S$-maximal prefix code.

Similarly, for $w\in S$,
we  let $\Gamma'_S(w)=\{x\in S\mid xw\in S\cap wA^+\}$ denote
the set of \emph{left return words} to $w$ and
$\RR'_S(w)=\Gamma'_S(w)\setminus A^+\Gamma'_S(w)$  the 
set of \emph{first left  return words} to $w$.
By definition, the set $\RR'_S(w)$ is, for every $w\in S$, a suffix code.
If $S$ is recurrent, it is an $Sw^{-1}$-maximal suffix code.
 The relation between $\RR_S(w)$
and $\RR'_S(w)$ is simply
\begin{equation}
  w\RR_S(w)=\RR'_S(w)w\,.\label{eqAutomo}
\end{equation}
Let $f:B^*\rightarrow A^*$ be a coding morphism for $\RR_S(w)$. The morphism
$f':B^*\rightarrow A^*$ defined for $b\in B$ by $f'(b)w=wf(b)$
is a coding morphism for $\RR'_S(w)$ called the coding morphism
\emph{associated} with $f$.
\begin{example}
Let $S$ be the uniformly recurrent set of Example~\ref{exampleDivision}.
We have
\begin{displaymath}
\RR_S(a)=\{cbba,ccba,ccbba\},\
\RR_S(b)=\{acb,accb,b\},\
\RR_S(c)=\{bac,bbac,c\}.
\end{displaymath}
These sets can be read from the word $y$ given in Equation~\eqref{Eqy}.
A coding morphism $f:B^*\rightarrow A^*$ with $B=A$ for the set $\RR_S(c)$ is given
by $f(a)=bac$, $f(b)=bbac$, $f(c)=c$.
\end{example}
Note that $\Gamma_S(w)\cup\{1\}$ is right unitary and that
\begin{equation}
\Gamma_S(w)\cup\{1\}=\RR_S(w)^*\cap w^{-1}S. \label{eqGamma1}
\end{equation}
Indeed, if $x\in\Gamma_S(w)$ is not in $\RR_S(w)$, we have $x=zu$
with $z\in\Gamma_S(w)$ and $u$ nonempty. Since $\Gamma_S(w)$ is right
unitary, we have $u\in \Gamma_S(w)$, whence the conclusion by induction
on the length of $x$. The converse inclusion is obvious.
\begin{proposition}\label{propReturnsFinite}
A recurrent set $S$ is uniformly recurrent if and only if
the set $\RR_S(w)$ is finite for all $w\in S$.
\end{proposition}
\begin{proof} Assume that all sets $\RR_S(w)$ for $w\in S$ are finite.
Let $n\ge 1$.
Let $N$
be the maximal length of the words in $\RR_S(w)$ for a word $w$
of length $n$. Then any word of length $N+n$ contains an occurrence
of $w$ . Indeed, assume that $u$ is a word of length $N+n$
without factor equal to $w$.
Let $r$ be a word  of minimal length such that
$ru$ begins with $w$ and set $ru=ws$. Then $|s|\ge N$ although
$s$ is a proper prefix of a word in $\RR(w)$, a contradiction. Conversely, for $w\in S$, let $N$ be such that $w$ is a
factor of any word
in $S$ of length $N$. Then the words of $\RR_S(w)$ have length at most
$N$.
\end{proof}
Let $S$ be a  recurrent set and let $w\in S$.
Let $f$ be a coding morphism for $\RR_S(w)$.
The set $f^{-1}(w^{-1}S)$, denoted $D_f(S)$, is called the \emph{derived set} of
$S$ with respect to $f$.  Note that if $f'$ is
the coding morphism for $\RR'_S(w)$ associated with $f$,
 then $D_f(S)=f'^{-1}(Sw^{-1})$.

The following result gives an equivalent definition of the derived set.
\begin{proposition}\label{propositionRecurrent}
Let $S$ be a recurrent set. For $w\in S$, let $f$ be a coding
morphism for the set $\RR_S(w)$. Then 
\begin{equation}
D_f(S)=f^{-1}(\Gamma_S(w))\cup \{1\}.\label{eqMagique}
\end{equation}
Moreover the set $D_f(S)$ is recurrent.
\end{proposition}
\begin{proof}
Let $z\in D_f(S)$. Then $f(z)\in w^{-1}S\cap R_S(w)^*$ and thus $f(z)\in \Gamma_S(w)\cup\{1\}$. Conversely, if $x\in \Gamma_S(w)$, then $x\in \RR_S(w)^*$
by Equation~\eqref{eqGamma1} and thus $x=f(z)$ for some $z\in D_f(S)$.
This proves~\eqref{eqMagique}.

Consider two nonempty words
$u,v\in D_f(S)$. By~\eqref{eqMagique}, we have $f(u),f(v)\in \Gamma_S(w)$. 
Since $S$ is recurrent, there is a word $t$ such that $wf(u)twf(v)\in S$.
Then $tw\in\Gamma_S(w)$ and thus $uf^{-1}(tw)v\in D_f(S)$ by~\eqref{eqMagique} again.
This shows that $D_f(S)$ is recurrent.
\end{proof}

Let $S$ be a recurrent set and $x$ be an infinite word such that
$S=\Fac(x)$. Let $w\in S$ and let $f$ be a coding morphism
for the set $\RR_S(w)$.
Since $w$ appears infinitely often in $x$, there
is a unique factorization $x=vwz$ with $z \in \RR_S(w)^\omega$
and $v$ such that $vw$ has no proper prefix ending with $w$.
The infinite word $f^{-1}(z)$ is called the \emph{derived word}
of $x$ relative to $f$, denoted $D_f(x)$. If $f'$ is the coding morphism
for $\RR'_S(w)$ associated with $f$, we have $f^{-1}(z)=f'^{-1}(wz)$
and thus $f,f'$ define the same derived word.

The following  statement results easily from Proposition~\ref{propositionRecurrent}.

\begin{proposition}\label{propositionDerived}
Let $S$ be a recurrent set and let $x$ be a recurrent infinite word such that
$S=\Fac(x)$. Let $w\in S$ and let $f$ be a coding morphism for 
 $\RR_S(w)$. The derived set of $S$ with respect to $f$
is the set of factors of the derived word of $x$ with respect to $f$,
that is, $D_f(S)=\Fac(D_f(x))$.
\end{proposition}

\begin{example}
Let $S$ be the uniformly recurrent set of Example~\ref{exampleDivision}.
Let $f$ be the coding morphism for the set $\RR_S(c)$ given
by $f(a)=bac$, $f(b)=bbac$, $f(c)=c$. Then the derived set of $S$
with respect to $f$
is represented in Figure~\ref{figureSetF} on the right.
\end{example}

\section{Uniformly recurrent tree sets}\label{sectionTreeNormal}
In this section, we recall the notion of tree set introduced
in~\cite{BertheDeFeliceDolceLeroyPerrinReutenauerRindone2013a}.
We recall that the factor complexity of a tree set
on $k+1$ letters is $p_n=kn+1$.

We recall a result concerning the decoding of tree sets
(Theorem~\ref{InverseImageTree}). We also recall the finite index
basis property of uniformly recurrent tree sets (Theorems~\ref{theoremFIB} and~\ref{theoremGroupCode})
that we will use in Section~\ref{sectionBifixDecoding}. We prove that the
family of uniformly recurrent tree sets
is closed under derivation (Theorem~\ref{propositionReturns}). We
further
prove that all bases of the free group included in a uniformly
recurrent tree set are tame (Theorem~\ref{theoremTame}).

\subsection{Tree sets}
Let $S$ be a fixed factorial set.
For a  word $w\in S$, we
 consider the undirected graph $G(w)$ on the set of vertices
which is the disjoint union of
$L(w)$ and $R(w)$ with edges the pairs 
$(a,b)\in E(w)$. The graph $G(w)$ is called the \emph{extension graph} of $w$
in $S$.
\begin{example}
Let $S$ be the Fibonacci set. The extension graphs of $\varepsilon,a,b,ab$ respectively are
shown in Figure~\ref{FigureExtensionGraph}.
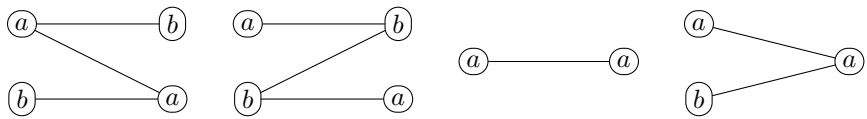
\begin{figure}[hbt]
\centering
\gasset{AHnb=0,Nadjust=wh}
\begin{picture}(110,12)
\put(0,0){
\begin{picture}(20,10)
\node(bL)(0,0){$b$}\node(aL)(0,10){$a$}
\node(bR)(20,10){$b$}\node(aR)(20,0){$a$}

\drawedge(bL,aR){}\drawedge(aL,bR){}\drawedge(aL,aR){}
\end{picture}
}
\put(30,0){
\begin{picture}(20,10)
\node(bL)(0,0){$b$}\node(aL)(0,10){$a$}
\node(bR)(20,10){$b$}\node(aR)(20,0){$a$}

\drawedge(bL,aR){}\drawedge(aL,bR){}\drawedge(bL,bR){}
\end{picture}
}
\put(60,0){
\begin{picture}(20,10)
\node(aL)(0,5){$a$}
\node(aR)(20,5){$a$}

\drawedge(aL,aR){}
\end{picture}
}
\put(90,0){
\begin{picture}(20,10)
\node(aL)(0,10){$a$}\node(bL)(0,0){$b$}
\node(aR)(20,5){$a$}

\drawedge(aL,aR){}
\drawedge(bL,aR){}
\end{picture}
}
\end{picture}
\caption{The extension graphs of $\varepsilon,a,b,ab$ in the Fibonacci set.}\label{FigureExtensionGraph}
\end{figure}

\end{example}

Recall that an undirected graph is a tree if it is connected and acyclic.

We say that $S$ is a \emph{tree set} (resp. an acyclic set)
if it is biextendable
and if for every  word
$w\in S$, the graph $G(w)$ is a tree (resp. is acyclic). 

It is not difficult to verify the following
statement (see~\cite[Proposition 3.3]{BertheDeFeliceDolceLeroyPerrinReutenauerRindone2013a}), which shows that the factor complexity of a tree set
is linear.
\begin{proposition} \label{propositionComplexity}
Let $S$ be a tree set on the alphabet $A$
and let $k=\Card(A\cap S)-1$. Then
$\Card(S\cap A^n)=kn+1$ for all $n\ge 0$.
\end{proposition}

The following result is also easy to prove.
\begin{proposition}\label{propositionSturmianisNormal}
A Sturmian set $S$ is a uniformly recurrent tree set.
\end{proposition}
\begin{proof} We have already seen that a Sturmian set is uniformly
  recurrent. Let us show that it is a tree set.
Consider $w\in S$. If $w$
is not left-special there is a unique $a\in A$ such that $aw\in S$.
Then $E(w)\subset \{a\}\times A$ and thus $G(w)$ is a tree. The
case where $w$ is not right-special is symmetrical. Finally, assume
that $w$ is bispecial. Let $a,b\in A$ be such that $aw$ is
right-special
and $wb$ is left-special. Then $E(w)= (\{a\}\times A)\cup (A\times \{b\})$
and thus $G(w)$ is a tree.
\end{proof}
Putting together Proposition~\ref{propositionRegularUR} and
 \cite[Proposition 4.2]{BertheDeFeliceDolceLeroyPerrinReutenauerRindone2013ie}, we have the similar statement.
\begin{proposition}\label{propositionExchangeTreeCondition}
A regular interval exchange set is a uniformly recurrent tree set.
\end{proposition}

Proposition~\ref{propositionExchangeTreeCondition} is actually a particular case of a result
of~\cite{FerencziZamboni2008} which characterizes the 
 regular interval exchange sets.

We give two examples of a uniformly recurrent tree set which is neither a Sturmian
set nor an interval exchange set. The first one is a maximal bifix decoding
of a Sturmian set (see Example~\ref{exampleTribonacci2} below).

\begin{example}\label{exampleTribonacci21}
Let $S$ be the Tribonacci set on the alphabet $A=\{a,b,c\}$
(see Example~\ref{exampleTribonacci}). Let $X=A^2\cap S$.
Then $X=\{aa,ab,ac,ba,ca\}$ is an $S$-maximal bifix code
of $S$-degree $2$. Let $B=\{x,y,z,t,u\}$ and let
$f:B^*\rightarrow A^*$ be the morphism defined by $f(x)=aa$,
$f(y)=ab$, $f(z)=ac$, $f(t)=ba$, $f(u)=ca$. Then $f$ is a coding
morphism for $X$. 
We will see that the set $T=f^{-1}(S)$ is a uniformly recurrent tree
 set (this follows from Theorem~\ref{theoremNormal} below).
It is not Sturmian since $y$ and $t$ are two right-special words
of length $1$. It is neither an interval exchange set.
 Indeed, for every right-special word $w$ of $T$, one has $r(w) = 3$. This is not
possible in a regular interval exchange set since, $\Sigma_T$  the length
of the intervals $J_w$ tends to $0$ as $|w|$ tends to infinity.
This implies that  any long enough right-special word $w$ is
such that $r(w)=2$.
\end{example}
The second example is a fixed point of a morphism obtained using
$S$-adic representations of tree sets (see Section~\ref{sectionSadic} below).

\begin{example}\label{exampleJulienLeroy}
Let $A=\{a,b,c\}$ and let $f$ be the morphism from $A^*$ into itself
defined by $f(a)=ac$, $f(b)=bac$, $f(c)=cbac$. Let $S$ be the
set of factors of $f^\omega(a)$. Since $f$ is primitive, $S$ is uniformly
recurrent. The right-special words
are the suffixes of the words $f^n(c)$ for $n\ge 1$ and the left-special words
are the prefixes of the words $f^n(a)$ or $f^n(c)$ for $n\ge 1$, 
as one may verify. 
Any right-special word $w$ is such that $r(w)=3$
and thus $S$ is not an interval exchange set. There are two
left-special words of each length and thus $S$ is not a Sturmian
 set.  Let us show by induction on the length of $w$
that for any bispecial word $w\in S$, the graph $G(w)$ is a tree.
It is true for $w=c$ and $w=ac$. Assume that $|w|\ge 2$. Either
$w$ begins with $a$ or with $c$. Assume the first case. Then $w$ begins
and ends with $ac$. We must have $w=acf(u)$ where $u$ is a bispecial word
beginning and ending with $c$.  In the second case, $w$ begins with
$cbac$ and ends with $ac$. We must have $w=cbacf(u)$ where
$u$ is a bispecial word beginning with $a$. In both cases,
by induction hypothesis, $G(u)$ is a tree
and thus $G(w)$ is a tree. This method
for computing the bispecial factors has been developed for a large
class of morphisms
in~\cite{Klouda2012}, inspired by Cassaigne's work~\cite{Cassaigne1997}.
The fact that $S$ is a tree set is also a consequence of
the results of~\cite{Leroy2014bis}.
\end{example}

Let $S$ be a set of words. For $w\in S$, and $U,V\subset S$,
let 
$U(w)=\{\ell\in U\mid \ell w\in S\}$
 and let $V(w)=\{r\in V\mid wr\in S\}$.
The \emph{generalized extension graph} of $w$ relative to
$U,V$ is the following undirected graph $G_{U,V}(w)$. The set of vertices is
made of two disjoint copies of $U(w)$ and $V(w)$.
 The edges are the pairs $(\ell,r)$
for $\ell\in U(w)$ and $r\in V(w)$
such that $\ell wr\in S$. The extension graph $G(w)$ defined previously
corresponds
to the case where $U,V=A$.

The following result is proved in~\cite[Proposition 3.9]{BertheDeFeliceDolceLeroyPerrinReutenauerRindone2013a}.
\begin{proposition}\label{PropStrongTreeCondition}
Let $S$ be a tree set.  For any $w\in S$, any
 finite 
$S$-maximal suffix code  $U\subset S$ 
and any  finite  $S$-maximal prefix code $V\subset S$,
 the generalized extension graph $G_{U,V}(w)$ is  a tree.
\end{proposition}

Let $S$ be a recurrent set and let $f$ be a coding morphism for a
finite $S$-maximal bifix code. The set $f^{-1}(S)$ is called
a \emph{maximal bifix decoding} of $S$.

The  following result is  in~\cite[Theorem 3.13]{BertheDeFeliceDolceLeroyPerrinReutenauerRindone2013a}.
\begin{theorem}\label{InverseImageTree}
Any maximal bifix decoding of a recurrent tree set is a tree set.
\end{theorem}

We have no example of a maximal bifix decoding of a recurrent tree
set which is not recurrent (in view of Theorem~\ref{theoremNormal}
to be proved hereafter, such a set would be the decoding of a
recurrent
tree set which is not uniformly recurrent).

\subsection{The finite index basis property}\label{sectionNormal}

Let $S$ be a recurrent set containing the alphabet $A$.
We say that $S$ has the \emph{finite index basis property} if the
following holds.
A  finite bifix code $X\subset S$ is an  $S$-maximal bifix code of $S$-degree
$d$ if and only if it is a basis of a subgroup of index $d$ of the
free group on $A$.

We recall the main result of
~\cite[Theorem 4.4]{BertheDeFeliceDolceLeroyPerrinReutenauerRindone2013}.
\begin{theorem}\label{theoremFIB}
A uniformly recurrent tree set containing the alphabet $A$
has the finite index basis property.
\end{theorem}
Recall from Section~\ref{subsectionautomata}
that a \emph{group code} of degree $d$ is a bifix code $X$ such that
$X^*=\varphi^{-1}(H)$ for a surjective morphism
$\varphi:A^*\rightarrow G$ from $A^*$ onto a finite group $G$
and a subgroup $H$ of index $d$ of $G$.

We will use the following result. It is stated for a Sturmian set $S$
in~\cite[Theorem 7.2.5]{BerstelDeFelicePerrinReutenauerRindone2012} 
but the proof only uses the fact that $S$ is uniformly
recurrent and satisfies the finite index basis property.
We reproduce the proof for the sake of clarity.

For a set of words $X$, we denote by  $\langle X\rangle$ the subgroup
of the free group on $A$ generated by $X$. The free group on $A$
itself is  denoted  $F_A$.

\begin{theorem}\label{theoremGroupCode}
Let $Z\subset A^+$ be a group code of degree $d$.
For every  uniformly recurrent tree set $S$ containing the alphabet $A$,
 the set $X=Z\cap S$ is a basis
of a subgroup of index $d$ of $F_A$. 
\end{theorem}
\begin{proof}
By~\cite[Theorem 4.2.11]{BerstelDeFelicePerrinReutenauerRindone2012},
the code $X$ is an $S$-maximal bifix code of $S$-degree $e\le d$. 
Since $S$ is a uniformly recurrent, by~\cite[Theorem
4.4.3]{BerstelDeFelicePerrinReutenauerRindone2012}, $X$ is
finite.
 By Theorem~\ref{theoremFIB},
$X$ is a basis of a subgroup of index
$e$.
Since $\langle X\rangle\subset \langle Z\rangle$, the index $e$
of the subgroup $\langle X\rangle$ is a multiple of the index $d$
of the subgroup $\langle Z\rangle$. Since $e\le d$,
this implies that $e=d$.
\end{proof}

As an example of this result, if $S$ is a uniformly recurrent tree set, then
$S\cap A^n$ is a basis of the subgroup  of the free group which is the
kernel of the morphism onto $\Z/n\Z$ sending any letter to $1$. 

We will use the following results
from~\cite{BertheDeFeliceDolceLeroyPerrinReutenauerRindone2013a}.
The first one is~\cite[Theorem 4.5]{BertheDeFeliceDolceLeroyPerrinReutenauerRindone2013a}.
\begin{theorem}\label{theoremJulien}
Let $S$ be a uniformly recurrent tree set containing the
alphabet $A$. For any word $w\in S$,
the set $\RR_S(w)$ is a basis of the free group on $A$.
\end{theorem}

The next result is  
\cite[Theorem 5.2]{BertheDeFeliceDolceLeroyPerrinReutenauerRindone2013a}.
A submonoid $M$ of $A^*$ is \emph{saturated} in a set $S$ if
$M\cap S=\langle M\rangle\cap S$.

\begin{theorem}\label{propositionHcapF}
  Let $S$ be an acyclic set. The submonoid generated by any bifix code
  $X\subset S$ is saturated in $S$.
\end{theorem}
\subsection{Derived sets of tree sets}

We will use the following closure property
of the family of uniformly recurrent tree sets.
It generalizes the fact that the derived word of a Sturmian
word
is Sturmian (see~\cite{JustinVuillon2000}).
\begin{theorem}\label{propositionReturns}
Any derived set of a uniformly recurrent tree set
 is a uniformly recurrent tree set.
\end{theorem}
\begin{proof}
Let $S$ be a uniformly recurrent tree set containing $A$, let $v\in S$
and let $f$ be a coding morphism for
 $X=\RR_S(v)$.
By Theorem~\ref{theoremJulien}, $X$ is a basis of the free group on $A$. Thus
$f:B^*\rightarrow A^*$ extends to an isomorphism from
$F_B$ onto $F_A$.

Set $H=f^{-1}(v^{-1}S)$. By
Proposition~\ref{propositionRecurrent},
 the set $H$ is
recurrent and $H=f^{-1}(\Gamma_S(v))\cup \{1\}$.

Consider $x\in H$ and set $y=f(x)$. Let
$f'$
be the  coding morphism for $X'=\RR'_S(v)$ associated with $f$.
For $a,b\in B$, we have
\begin{displaymath}
(a,b)\in G(x)\Leftrightarrow (f'(a),f(b))\in G_{X',X}(vy),
\end{displaymath}
where $G_{X',X}(vy)$ denotes the generalized extension graph of $vy$
relative to $X',X$.
Indeed, 
\begin{displaymath}
axb\in H
\Leftrightarrow f(a)yf(b)\in \Gamma_S(v)
\Leftrightarrow vf(a)yf(b)\in S
\Leftrightarrow
f'(a)vyf(b)\in S.
\end{displaymath}
The set $X'$ is an $Sv^{-1}$-maximal suffix code and the set $X$ is a
$v^{-1}S$-maximal prefix code. By
Proposition~\ref{PropStrongTreeCondition}
the generalized extension graph $G_{X',X}(vy)$ is a tree. Thus the
graph $G(x)$
is a tree. This shows that $H$ is a tree set.

Consider now $x\in H\setminus 1$. Set $y=f(x)$.
 Let us show that $\Gamma_H(x)=f^{-1}(\Gamma_S(vy))$
or equivalently $f(\Gamma_H(x))=\Gamma_S(vy)$. 
Consider first $r\in \Gamma_H(x)$. Set $s=f(r)$. Then $xr=ux$ with
$u,ux\in H$. Thus $ys=wy$ with $w=f(u)$. 

Since $u\in H\setminus \{1\}$,  $w=f(u)$ is in $\Gamma_S(v)$,
we have $vw\in A^+v\cap S$. This implies that $vys=vwy\in A^+vy\cap S$
and thus that $s\in \Gamma_S(vy)$. 
Conversely, consider $s\in
 \Gamma_S(vy)$. Since $y=f(x)$, we have $s\in\Gamma_S(v)$.
Set $s=f(r)$.
Since $vys\in A^+vy\cap S$, we have $ys\in A^+y\cap
 S$.
Set $ys=wy$. Then $vwy\in A^+vy$ implies $vw\in A^+v$ and therefore
$w\in \Gamma_S(v)$. Setting $w=f(u)$, we obtain 
$f(xr)=ys=wy\in X^+y\cap \Gamma_S(v)$. Thus $r\in\Gamma_H(x)$.
 This shows that $f(\Gamma_H(x))=\Gamma_S(vy)$ and
thus that $\RR_H(x)=f^{-1}(\RR_S(vy))$.

Since
$S$ is uniformly recurrent, the set $\RR_S(vy)$ is finite.
Since $f$ is an isomorphism, $\RR_H(x)$ is also finite,
 which shows that $H$ is uniformly recurrent.
\end{proof}

\begin{example}
Let $S$ be the Tribonacci set (see Example~\ref{exampleTribonacci}).
It is the set of factors of the infinite word $x=abacaba\cdots$
which is the fixed point of the morphism $f$ defined by $f(a)=ab$,
$f(b)=ac$, $f(c)=a$.
We have $\RR_S(a)=\{a,ba,ca\}$.  Let  $g$ be the coding morphism for
$\RR_S(a)$ defined by  $g(a)=a$, $g(b)=ba$,
$g(c)=ca$ and let $g'$ be the associated coding
morphism for $\RR'_S(a)$. We have $f=g'\pi$ where
$\pi$ is the circular permutation $\pi=(abc)$.
Set  $z=g'^{-1}(x)$. Since $g'\pi(x)=x$, we have $z=\pi(x)$. 
Thus the derived set of $S$ with respect
to $a$ is the set $\pi(S)$.
\end{example}
\subsection{Tame bases}
An automorphism $\alpha$ of the free group on $A$ is \emph{positive}
if $\alpha(a)\in A^+$ for every $a\in A$.
We say that a positive automorphism of the free group on $A$ is \emph{tame}\footnote{
The word \emph{tame} (as opposed to \emph{wild}) is used here on analogy
with its use in ring theory (see~\cite{Cohn1985}). The tame
automorphisms as introduced here should, strictly speaking, be called
positive tame automophisms since the group of all automorphisms,
positive or not, is tame in the sense that it is generated by the
elementary automorphisms.}
if it belongs to the submonoid generated by the permutations of $A$ and
the automorphisms $\alpha_{a,b}$, $\tilde{\alpha}_{a,b}$ defined 
for $a,b\in A$ with $a\ne b$ by
\begin{displaymath}
\alpha_{a,b}(c)=\begin{cases}ab&\text{ if $c=a$,}\\c&\text{
    otherwise} \end{cases}
\quad
\text{ and }
\quad
\tilde{\alpha}_{a,b}(c)=\begin{cases}ba&\text{ if $c=a$,}\\c&\text{
    otherwise.} \end{cases}
\end{displaymath}
Thus $\alpha_{a,b}$ places a letter $b$ after each $a$ and
 $\tilde{\alpha}_{a,b}$ places a letter $b$ before each $a$.
The above automorphisms and the permutations of $A$
are called the \emph{elementary}
positive automorphisms on $A$.
The monoid of positive automorphisms is not finitely generated
as soon as the alphabet has at least three generators
(see~\cite{TanWenZhang2004}).

A basis $X$ of the free group is \emph{positive} if $X\subset A^+$.
A positive basis $X$ of the free group is \emph{tame} if there exists
a tame automorphism $\alpha$ such that $X=\alpha(A)$.

\begin{example}
The set $X=\{ba,cba,cca\}$ is a tame basis of the free group on $\{a,b,c\}$.
Indeed, one has the following sequence of elementary automorphisms.
\begin{displaymath}
(b,c,a)\edge{\alpha_{c,b}}(b,cb,a)\edge{\tilde{\alpha}_{a,c}^2}(b,cb,cca)
\edge{\alpha_{b,a}}(ba,cba,cca).
\end{displaymath}
The fact that $X$ is a basis can be checked directly by the fact that
$(cba)(ba)^{-1}=c$, $c^{-2}(cca)=a$ and finally $(ba)a^{-1}=b$.
\end{example}
The following result will play a key role in the proof of the
main result of this section (Theorem~\ref{theoremTame}).
\begin{proposition}\label{propAuxiliary}
A set $X\subset A^+$ is a tame basis of the free group on $A$
if and only if $X=A$ or there is a tame basis $Y$ of the free
group on $A$ and $u,v\in Y$ such that
$X=(Y\setminus v)\cup uv$ or $X=(Y\setminus u)\cup uv$.
\end{proposition}
\begin{proof}
Assume first that $X$ is a tame basis of the free group on $A$.
Then $X=\alpha(A)$ where $\alpha$ is a tame automorphism
of $\langle A\rangle$. Then $\alpha=\alpha_1\alpha_2\cdots\alpha_n$
where the $\alpha_i$ are elementary positive automorphisms.
We use an induction on $n$. If $n=0$, then $X=A$. 
If $\alpha_n$ is a permutation of $A$, then
$X=\alpha_1\alpha_2\cdots\alpha_{n-1}(A)$ and the result holds by
induction
hypothesis. Otherwise, set $\beta=\alpha_1\cdots\alpha_{n-1}$ and
 $Y=\beta(A)$. By induction hypothesis, $Y$ is tame.
If $\alpha_n=\alpha_{a,b}$, set $u=\beta(a)$ and $v=\beta(b)=\alpha(b)$.
Then $X=(Y\setminus u)\cup uv$ and thus the condition is satisfied.
The case were $\alpha_n=\tilde{\alpha}_{a,b}$ is symmetrical.

Conversely, assume that $Y$ is a tame basis and that $u,v\in Y$
are such that $X=(Y\setminus u)\cup uv$. Then, there is a tame automorphism
$\beta$ of $\langle A\rangle$ such that $Y=\beta(A)$. Set $a=\beta^{-1}(u)$
and $b=\beta^{-1}(v)$. Then $X=\beta\alpha_{a,b}(A)$ and thus $X$ is a tame
basis.
\end{proof}
We note the following corollary. 
\begin{corollary}\label{corollaryTame}
A tame basis of the free group which is a bifix code is the alphabet.
\end{corollary}
\begin{proof}
Assume that $X$ is a tame basis which is not the alphabet.
By Proposition~\ref{propAuxiliary} there is a tame basis $Y$ and $u,v\in Y$
such that $X=(Y\setminus v)\cup uv$ or  $X=(Y\setminus u)\cup uv$.
In the first case, $X$ is not prefix. In the second one, it is not suffix.
\end{proof}
The following example is from~\cite{TanWenZhang2004}.
\begin{example}\label{exampleWen}
The set $X=\{ab,acb,acc\}$ is a basis of the free group
on $\{a,b,c\}$. Indeed, $accb=(acb)(ab)^{-1}(acb)\in\langle X\rangle$
and thus $b=(acc)^{-1}accb\in \langle X\rangle$, which implies easily
that $a,c\in\langle X\rangle$.
 The set $X$ is bifix and thus it is not a tame basis by
Corollary~\ref{corollaryTame}.
\end{example}
The following result is a remarkable consequence of Theorem~\ref{theoremFIB}.
\begin{theorem}\label{theoremTame}
Any basis of the free group included in a uniformly recurrent tree set
is tame.
\end{theorem}
\begin{proof}
Let $S$ be a uniformly recurrent tree set. Let $X\subset S$
be a basis of the free group on $A$. Since $A$ is finite, $X$
is finite (and of the same cardinality as $A$). We use an induction
on the sum $\lambda(X)$ of the lengths of the words of $X$. If $X$
is bifix, by Theorem~\ref{theoremFIB}, it is an $S$-maximal
bifix code of $S$-degree $1$. Thus $X=A$ (see
Example~\ref{exampleDegree1}).
Next assume for example that $X$ is not prefix. Then there
are nonempty words $u,v$ such that $u,uv\in X$. Let
$Y=(X\setminus uv)\cup v$. Then $Y$ is a basis of the free
group and $\lambda(Y)<\lambda(X)$. By induction
hypothesis, $Y$ is tame. Since $X=(Y\setminus v)\cup uv$,
$X$ is tame
by Proposition~\ref{propAuxiliary}.
\end{proof}
\begin{example}
The set $X=\{ab,acb,acc\}$ is a basis of the free group which is not
tame (see Example~\ref{exampleWen}). Accordingly, the extension
graph $G(\varepsilon)$ relative to the set of factors
of $X$ is not a tree (see Figure~\ref{figureWen}).
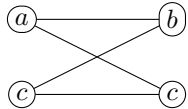
\begin{figure}[hbt]
\centering
\gasset{Nadjust=wh,AHnb=0}
\begin{picture}(20,10)
\node(al)(0,10){$a$}\node(cl)(0,0){$c$}
\node(br)(20,10){$b$}\node(cr)(20,0){$c$}
\drawedge(al,br){}\drawedge(al,cr){}
\drawedge(cl,br){}\drawedge(cl,cr){}
\end{picture}
\caption{The graph $G(\varepsilon)$.}\label{figureWen}
\end{figure}
\end{example}
\subsection{$S$-adic representations}\label{sectionSadic}
In this section we study $S$-adic representations of tree sets.
This notion was introduced in~\cite{Ferenczi1996}, using a terminology
initiated by Vershik and coined out by B.~Host.
We first recall a general construction allowing to build $S$-adic representations of any uniformly recurrent aperiodic set (Proposition~\ref{prop: S-adic UR set}) which is based on return words.
Using Theorem~\ref{theoremTame}, we show that this construction actually 
provides $\mathcal{S}_e$-representations of uniformly recurrent tree sets (Theorem~\ref{base tame}), where $\mathcal{S}_e$ is the set of elementary positive automorphisms
 of the free group on $A$.

Let $S$ be a set of morphisms and $\mathbf{h} = (\sigma_n)_{n \in \N}$ 
be a sequence in $S^\N$ with $\sigma_n:A_{n+1}^* \to A_n^*$ and $A_0=A$.
We let $T_\mathbf{h}$ denote the set of words 
$\bigcap_{n \in \N} \Fac(\sigma_0 \cdots \sigma_n(A_{n+1}^*))$.
We call a factorial set $T$ an {\em $S$-adic set} if there exists 
$\mathbf{h} \in S^\N$ such that $T = T_\mathbf{h}$. 
In this case, the sequence $\mathbf{h}$ is called an 
{\em $S$-adic representation} of $T$.

\begin{example}
Any Sturmian set is $S$-adic with a finite set $S$. This results from the fact that any Sturmian
word is obtained by iterating a sequence of morphism of the form
$\psi_a$ for $a\in A$ defined by $\psi_a(a)=a$ and $\psi_a(b)=ab$
for $b\ne a$ (see~\cite{ArnouxRauzy1991} 
or~\cite{BerstelDeFelicePerrinReutenauerRindone2012}). 
\end{example}

A sequence of morphisms $(\sigma_n)_{n \in \N}$ is said to be 
{\em everywhere growing} if $\min_{a\in A_n}$ $|\sigma_0 \cdots \sigma_{n-1}(a)|$ 
goes to infinity as $n$ increases.
A sequence of morphisms $(\sigma_n)_{n \in \N}$ is said to be {\em primitive} if for all $r \geq 0$ 
there exists $s >r$ such that all letters of $A_r$ occur in all images $\sigma_r \cdots \sigma_{s-1}(a)$, $a \in A_s$.
Obviously any primitive sequence of morphisms is everywhere growing.

A uniformly recurrent set $T$ is said to be {\em aperiodic} if it contains at least one right-special factor of 
each length.
The next (well-known) proposition provides a general construction to get a primitive $S$-adic representation 
of any aperiodic uniformly recurrent set $T$.

\begin{proposition}
\label{prop: S-adic UR set}
An aperiodic factorial set $T \subset A^*$ is uniformly recurrent if and only if it has a primitive $S$-adic 
representation for some (possibly infinite) set $S$ of morphisms.
\end{proposition}
\begin{proof}
Let $S$ be a set of morphisms and  $\mathbf{h} = (\sigma_n:A_{n+1}^* \to A_n^*)_{n \in \N} \in S^\N$ 
be a primitive sequence of morphisms such that $T = \bigcap_{n \in \N} \Fac(\sigma_0 \cdots \sigma_n(A_{n+1}^*))$.
Consider a word $u \in T$ and let us prove that $u \in \Fac(v)$ 
for all long enough $v \in T$.
The sequence $\bh$ being everywhere growing, there is an integer $r > 0$ such that 
$\min_{a \in A_r} |\sigma_0 \cdots \sigma_{r-1}(a)| > |u|$.
As $T = \bigcap_{n \in \N} \Fac(\sigma_0 \cdots \sigma_n(A_{n+1}^*))$, there is an integer $s > r$, 
two letters $a,b \in A_r$ and a letter $c \in A_s$ such that $u \in \Fac(\sigma_0 \cdots \sigma_{r-1}(ab))$ 
and $ab \in \Fac(\sigma_r \cdots \sigma_{s-1}(c))$.
The sequence $\bh$ being primitive, there is an integer $t > s$ such that $c$ occurs in 
$\sigma_s \cdots \sigma_{t-1}(d)$ for all $d \in A_t$.
Thus $u$ is a factor of all words $v \in T$ such that $|v| \geq 2\max_{d \in A_t} |\sigma_0 \cdots \sigma_{t-1}(d)|$ 
and $T$ is uniformly recurrent.

Let us prove the converse.
Let $(u_n)_{n \in \N} \in T^\N$ be a non-ultimately periodic sequence such that $u_n$ is suffix of $u_{n+1}$.
By assumption, $T$ is uniformly recurrent so $\mathcal{R}_T(u_{n+1})$ is finite for all $n$. 
The set $T$ being aperiodic, $\mathcal{R}_T(u_{n+1})$ also has cardinality at least 2 for all $n$.
For all $n$, let $A_n = \{0,\dots,\Card(\mathcal{R}_T(u_n))-1\}$ and let $\alpha_n: A_n^* \to A^*$ be a 
coding morphism for $\mathcal{R}_T(u_n)$.
The word $u_n$ being suffix of $u_{n+1}$, we have $\alpha_{n+1}(A_{n+1}) \subset \alpha_n(A_n^+)$.
Since $\alpha_n(A_n) = \mathcal{R}_T(u_n)$ is a prefix code, there is a unique 
morphism $\sigma_n: A_{n+1}^* \to A_n^*$ such that $\alpha_n \sigma_n = \alpha_{n+1}$.
For all $n$ we get 
$\mathcal{R}_T(u_n) = \alpha_0 \sigma_0 \sigma_1 \cdots \sigma_{n-1}(A_n)$ 
and $T = \bigcap_{n \in \N} \Fac(\alpha_0 \sigma_0 \cdots \sigma_n(A_{n+1}^*))$.
Without loss of generality, we can suppose that $u_0 = \varepsilon$ and $A_0 = A$.
In that case we get $\alpha_0 = \id$ and the set $S$ thus has an $S$-adic 
representation with $S = \{\sigma_n \mid n \in \N\}$.

Let us show that $\bh = (\sigma_n)_{n \in \N}$ is everywhere growing.
If not, there is a sequence of letters $(a_n \in A_n)_{n \geq N}$ such that 
$\sigma_n(a_{n+1}) = a_n$ for all $n \geq N$ for some $N\ge 1$.
This means that the word $v = \sigma_0 \cdots \sigma_n(a_n) \in T$ is a first 
return word to $u_n$ for all $n \geq N$.
The sequence $(|u_n|)_{n \in \N}$ being unbounded, the word $v^k$ belongs to $T$ 
for all positive integers $k$, which contradicts the uniform recurrence of $T$.

Let us show that $\bh$ is primitive.
The set $T$ being uniformly recurrent, for all $n \in \N$ there exists $N_n$ 
such that all words of $T \cap A^{\leq n}$ occur in all words of $T \cap A^{\geq N_n}$.
Let $r \in \N$ and let $u = \sigma_0 \cdots \sigma_{r-1}(a)$ for some $a \in A_r$.
Let $s>r$ be an integer such that
 $\min_{b \in A_s} |\sigma_0 \cdots \sigma_{s-1}(b)| \geq N_{|u|}$.
Thus $u$ occurs in $\sigma_0 \cdots \sigma_{s-1}(b)$ for all $b \in A_s$.
As $\sigma_0 \cdots \sigma_{s-1}(A_s) \subset \sigma_0 \cdots \sigma_{r-1}(A_r^+)$ 
and as $\sigma_0 \cdots \sigma_{r-1}(A_r) = \mathcal{R}_T(u_r)$ is a prefix code, 
the letter $a \in A_r$ occurs in $\sigma_r \cdots \sigma_{s-1}(b)$ for all $b \in A_r$.
\end{proof}

Even for uniformly recurrent sets with linear factor complexity, the set of morphisms $S=\{\sigma_n \mid n \in \N \}$ 
considered in Proposition~\ref{prop: S-adic UR set} is usually  infinite as well as the sequence of alphabets 
$(A_n)_{n \in \N}$ is usually  unbounded (see~\cite{Durand-Leroy-Richomme}).
For tree sets $T$, the next theorem significantly improves the only if part of 
Proposition~\ref{prop: S-adic UR set}: For such sets, the set $S$ can be replaced by the set $\mathcal{S}_e$ 
of elementary positive automorphisms. In particular, $A_n$ is equal to $A$ for all $n$.

\begin{theorem}\label{base tame}
If $T$ is a uniformly recurrent tree set over an alphabet $A$, then it has a primitive $\mathcal{S}_e$-adic representation.
\end{theorem}

\begin{proof}
For any non-ultimately periodic sequence $(u_n)_{n \in \N} \in T^\N$ such that $u_0 = \varepsilon$ and $u_n$ is suffix of $u_{n+1}$, the sequence of morphisms $(\sigma_n)_{n \in \N}$ built in the proof of Proposition~\ref{prop: S-adic UR set} is a primitive $S$-adic representation of $T$ with $S = \{\sigma_n \mid n \in \N\}$.
Therefore, all we need to do is to consider such a sequence $(u_n)_{n \in \N}$ such that $\sigma_n$ is tame for all $n$.

Let $u_1 = a^{(0)}$ be a letter in $A$. Set $A_0=A$
and let  $\sigma_0: A_1^* \to A_0^*$
be a coding morphism for $\mathcal{R}_T(u_1)$.
By Theorem~\ref{theoremJulien}, the set $\mathcal{R}_T(u_1)$ is a basis of the free group on $A$. 
By Theorem~\ref{theoremTame}, the morphism $\sigma_0: A_1^* \to A_0^*$ is tame ($A_0 = A$).
Let $a^{(1)} \in A_1$ be a letter and set $u_2 = \sigma_0(a^{(1)})$. 
Thus $u_2 \in \mathcal{R}_T(u_1)$ and $u_1$ is a suffix of $u_2$.
By Theorem~\ref{propositionReturns}, the derived set $T^{(1)} = \sigma_0^{-1}(S)$ is a uniformly recurrent tree set on the alphabet $A$. 
We thus reiterate the process with $a^{(1)}$ and we conclude by induction with $u_n = \sigma_0 \cdots \sigma_{n-2}(a^{(n-1)})$ for all $n \geq 2$.
\end{proof}
We illustrate Theorem~\ref{base tame} by the following example.
\begin{example}\label{exampleSadicTree}
Let $f$ and $S$ be as in Example~\ref{exampleJulienLeroy}. 
We have $f=\alpha_{a,c}\alpha_{b,a}\alpha_{c,b}$. Thus the tree set
$S$ has the $\mathcal{S}_e$-adic representation $(\sigma_n)_{n\ge 0}$ given
by the periodic sequence $\sigma_{3n}=\alpha_{a,c}$, $\sigma_{3n+1}=\alpha_{b,a}$, $\sigma_{3n+2}=\alpha_{c,b}$.
\end{example}
The converse of Theorem~\ref{base tame} is not true, as shown by Example~\ref{exampleNotTree} below. 
\begin{example}\label{exampleNotTree}
Let $A=\{a,b,c\}$ and let $f:a\mapsto ac, b\mapsto bac, c\mapsto cb$.
The set $S$ of factors of the fixed point $f^\omega(a)$ is not a tree set
since $bb,bc,cb,cc\in S$ and thus $G(\varepsilon)$ has a cycle although 
$f$ is a tame automorphism since $f=\alpha_{a,c}\alpha_{c,b}\alpha_{b,a}$.
\end{example}
In the case of a ternary alphabet, a characterization
of tree sets by their $S$-adic representation can be proved~\cite{Leroy2014bis},
showing that there exists a  B\"uchi automaton on the alphabet $\mathcal{S}_e$
recognizing
the set of  $S$-adic representations
of uniformly recurrent tree sets.

\section{Maximal bifix decoding}\label{sectionBifixDecoding}
In this section, we state and prove the main result of this paper
(Theorem~\ref{theoremNormal}). In the first part, we prove two results
concerning morphisms onto a finite group. In the second one we prove
a sequence of lemmas leading to a proof of the main result.
\subsection{Main result}\label{subsectionMainResult}
The family of uniformly recurrent tree sets contains both the Sturmian sets
and the regular interval exchange sets. The second
family is closed under maximal bifix decoding 
(see~\cite[Theorem 3.13]{BertheDeFeliceDolceLeroyPerrinReutenauerRindone2013ie})
but the first family is
not (see Example~\ref{exampleTribonacci2} below).
The following result shows that the family of uniformly recurrent tree
sets is a natural closure of the family of Sturmian sets.
\begin{theorem}\label{theoremNormal}
The family of uniformly recurrent tree sets is closed under maximal bifix decoding.
\end{theorem}
Thus, for any uniformly recurrent tree set and any coding morphism $f$
for a finite $S$-maximal bifix code, the set $f^{-1}(S)$ is a
uniformly recurrent tree set. This statement has a stronger hypothesis
than Theorem~\ref{theoremNormal}  and a stronger conclusion.

We illustrate Theorem~\ref{theoremNormal} by the following example.

\begin{example}\label{exampleTribonacci2}
Let $T$ be as in Example~\ref{exampleTribonacci21}. 
The set $T$ is a uniformly recurrent tree
 set by Theorem~\ref{theoremNormal}.
\end{example}
We prove two preliminary results concerning the restriction
to a uniformly recurrent tree
 set of a morphism onto a finite group (Propositions
\ref{propositionGroup} and~\ref{propGamma}).

\begin{proposition}\label{propositionGroup}
Let $S$ be a uniformly recurrent tree
 set containing the alphabet $A$ and let $\varphi:A^*\rightarrow G$
be a morphism from $A^*$
onto a finite group $G$. Then $\varphi(S)=G$.
\end{proposition}
\begin{proof}
Since the submonoid $\varphi^{-1}(1)$ is right and left unitary,
there is a bifix code $Z$ such that $Z^*=\varphi^{-1}(1)$.
Let $X=Z\cap S$. By Theorem~\ref{theoremGroupCode}, $X$ is a basis of a subgroup of index
$\Card(G)$.
Let $x$ be a word of $X$ of maximal length (since $X$ is a basis
of a subgroup of finite index,
it is finite). Then $x$ is not
an internal factor of $X$ and thus it has $\Card(G)$ parses.
Let $S(x)$ be the set of
 suffixes of $x$ which are prefixes of $X$. If $s,t\in S(x)$, then
they are comparable for the suffix order. Assume for example
that $s=ut$. If $\varphi(s)=\varphi(t)$, then $u\in X^*$ which
implies $u=1$ since $s$ is a prefix of $X$. Thus all elements
of $S(x)$ have distinct images by $\varphi$. Since $S(x)$ has
$\Card(G)$ elements, this forces $\varphi(S(x))=G$ and thus
$\varphi(S)=G$ since $S(x)\subset S$.
\end{proof}
We illustrate the proof on the following example.
\begin{example}
Let $A=\{a,b\}$ and let
$\varphi$ be the morphism from $A^*$ onto the symmetric group
$G$ on $3$ elements defined by $\varphi(a)=(12)$ and
$\varphi(b)=(13)$. Let $Z$ be the group code such that
$Z^*=\varphi^{-1}(1)$.
The group automaton corresponding to the regular
representation of $G$ is represented in
Figure~\ref{figGroupAutomaton} (this automaton has $G$
as set of states and $g\cdot a=g\varphi(a)$ for every $g\in G$ and $a\in A$).
 Let $S$ be the Fibonacci set. 
\begin{figure}[hbt]
\gasset{Nadjust=wh}\centering
\begin{picture}(60,30)(0,-3)
\node(beta)(0,20){$(13)$}\node(1)(20,20){$(1)$}\node(alpha)(40,20){$(12)$}
\node(alphabeta)(60,20){$(123)$}
\node(betaalpha)(20,0){$(132)$}\node(alphabetaalpha)(40,0){$(23)$}

\drawedge[curvedepth=3](beta,1){$b$}\drawedge[curvedepth=3,ELside=r](1,beta){$b$}
\drawedge[curvedepth=3](1,alpha){$a$}\drawedge[curvedepth=3](alpha,1){$a$}
\drawedge[curvedepth=3](alpha,alphabeta){$b$}\drawedge[curvedepth=3,ELside=r](alphabeta,alpha){$b$}
\drawedge[curvedepth=3](beta,betaalpha){$a$}\drawedge[curvedepth=3](betaalpha,beta){$a$}
\drawedge[curvedepth=3](betaalpha,alphabetaalpha){$b$}
\drawedge[curvedepth=3](alphabetaalpha,betaalpha){$b$}
\drawedge[curvedepth=3](alphabetaalpha,alphabeta){$a$}
\drawedge[curvedepth=3](alphabeta,alphabetaalpha){$a$}
\end{picture}
\caption{The group automaton corresponding to the regular
  representation of $G$.}\label{figGroupAutomaton}
\end{figure}
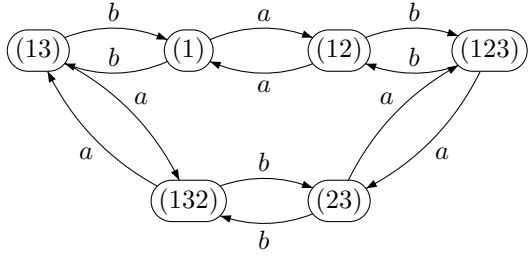
The code $X=Z\cap S$
is represented in Figure~\ref{figCodeX}.
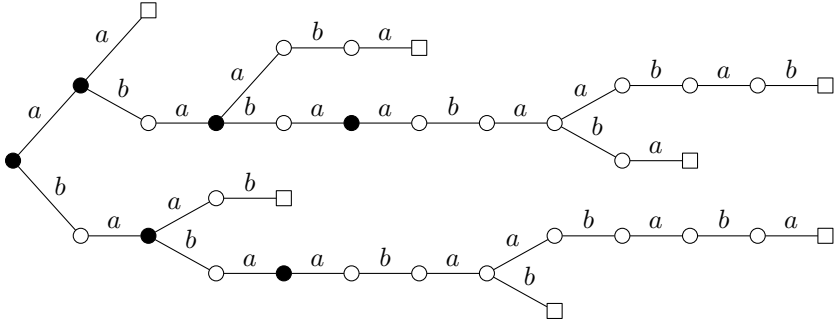
\begin{figure}[hbt]
\centering
\gasset{AHnb=0,Nadjust=wh}
\begin{picture}(130,45)
\node[Nfill=y](1)(0,20){}
\node[Nfill=y](a)(9,30){}\node(b)(9,10){}
\node[Nmr=0](aa)(18,40){}\node(ab)(18,25){}
\node[Nfill=y](ba)(18,10){}
\node[Nfill=y](aba)(27,25){}
\node(baa)(27,15){}\node(bab)(27,5){}
\node(abaa)(36,35){}\node(abab)(36,25){}
\node[Nmr=0](baab)(36,15){}\node[Nfill=y](baba)(36,5){}
\node(abaab)(45,35){}\node[Nfill=y](ababa)(45,25){}
\node(babaa)(45,5){}
\node[Nmr=0](abaaba)(54,35){}\node(ababaa)(54,25){}
\node(babaab)(54,5){}
\node(ababaab)(63,25){}
\node(babaaba)(63,5){}
\node(ababaaba)(72,25){}
\node(babaabaa)(72,10){}\node[Nmr=0](babaabab)(72,0){}
\node(ababaabaa)(81,30){}\node(ababaabab)(81,20){}
\node(babaabaab)(81,10){}
\node(ababaabaab)(90,30){}\node[Nmr=0](ababaababa)(90,20){}
\node(babaabaaba)(90,10){}
\node(ababaabaaba)(99,30){}
\node(babaabaabab)(99,10){}
\node[Nmr=0](ababaabaabab)(108,30){}
\node[Nmr=0](babaabaababa)(108,10){}

\drawedge(1,a){$a$}\drawedge(1,b){$b$}
\drawedge(a,aa){$a$}\drawedge(a,ab){$b$}
\drawedge(b,ba){$a$}
\drawedge(ab,aba){$a$}
\drawedge(ba,baa){$a$}\drawedge(ba,bab){$b$}
\drawedge(aba,abaa){$a$}\drawedge(aba,abab){$b$}
\drawedge(baa,baab){$b$}\drawedge(bab,baba){$a$}
\drawedge(abaa,abaab){$b$}\drawedge(abab,ababa){$a$}\drawedge(baba,babaa){$a$}
\drawedge(abaab,abaaba){$a$}\drawedge(ababa,ababaa){$a$}
\drawedge(babaa,babaab){$b$}
\drawedge(ababaa,ababaab){$b$}
\drawedge(babaab,babaaba){$a$}
\drawedge(ababaab,ababaaba){$a$}
\drawedge(babaaba,babaabaa){$a$}\drawedge(babaaba,babaabab){$b$}
\drawedge(ababaaba,ababaabaa){$a$}\drawedge(ababaaba,ababaabab){$b$}
\drawedge(babaabaa,babaabaab){$b$}
\drawedge(ababaabaa,ababaabaab){$b$}\drawedge(ababaabab,ababaababa){$a$}
\drawedge(babaabaab,babaabaaba){$a$}
\drawedge(ababaabaab,ababaabaaba){$a$}
\drawedge(babaabaaba,babaabaabab){$b$}
\drawedge(ababaabaaba,ababaabaabab){$b$}
\drawedge(babaabaabab,babaabaababa){$a$}
\end{picture}
\caption{The code $X=Z\cap S$.}\label{figCodeX}
\end{figure}
The word $w=ababa$ is not an internal factor of $X$. All its $6$
suffixes (indicated in black in Figure~\ref{figCodeX}) are proper prefixes of $X$ and their images by $\varphi$
are the $6$ elements of the group $G$.
\end{example}
\begin{proposition}\label{propGamma}
Let $S$ be a uniformly recurrent tree set containing the alphabet $A$ and let
$\varphi:A^*\rightarrow G$ be a morphism from $A^*$ onto a finite
group $G$. For any $w\in S$, one has $\varphi(\Gamma_S(w)\cup \{1\})=G$.
\end{proposition}
\begin{proof}
Let $\alpha:B^*\rightarrow A^*$ be a coding morphism for $\RR_S(w)$.
Then $\beta=\varphi\circ \alpha:B^*\rightarrow G$ is a morphism from
$B^*$ to $G$. By Theorem~\ref{theoremJulien}, 
the set $\RR_S(w)$ is a basis of the free group on $A$. Thus
$\langle\alpha(B)\rangle=F_A$. This
implies that $\beta(F_B)=G$. This implies that
$\beta(B)$ generates $G$. Since $G$
is a finite group,  $\beta(B^*)$ is a subgroup of $G$ and thus
$\beta(B^*)=G$.
By Theorem~\ref{propositionReturns}, the set
$H=\alpha^{-1}(w^{-1}S)$ is a uniformly recurrent tree set. Thus $\beta(H)=G$
by Proposition~\ref{propositionGroup}. This implies that
$\varphi(\Gamma_S(w)\cup \{1\})=G$.
\end{proof}

\subsection{Proof of the main result}

Let $S$ be a uniformly recurrent tree set containing $A$
and let $f:B^*\rightarrow A^*$ be a coding morphism for a finite
$S$-maximal bifix code $Z$. By Theorem~\ref{theoremFIB}, $Z$ is a basis
of a subgroup of index $d_Z(S)$ and, by Theorem~\ref{propositionHcapF},
the submonoid $Z^*$ is saturated in $S$.

We first prove the following lemma. 
\begin{lemma}\label{lemma1}
Let $S$ be a uniformly recurrent tree set containing $A$
 and let  $f:B^*\rightarrow A^*$ be
 a coding morphism for  an $S$-maximal
bifix code $Z$.
The set $T=f^{-1}(S)$ is recurrent. 
\end{lemma}
\begin{proof} Since $S$ is factorial, the set $T$ is factorial.
Let $r,s\in T$. Since $S$
is recurrent, there exists $u\in S$ such that $f(r)uf(s)\in S$. Set 
$t=f(r)uf(s)$.
Let $G$ be the representation of $F_A$ on the right cosets 
of $\langle Z\rangle$.
Let $\varphi:A^*\rightarrow G$ be the natural morphism from $A^*$ onto
$G$.
By Proposition~\ref{propGamma}, we have $\varphi(\Gamma_S(t)\cup \{1\})=G$.
Let $v\in\Gamma_S(t)$ be such that $\varphi(v)$ is the inverse
of $\varphi(t)$. Then $\varphi(tv)$ is the identity of $G$ and thus
$tv\in \langle Z\rangle$. 

Since $S$ is a tree set, it is acyclic and thus $Z^*$ is saturated in $S$
by Theorem~\ref{propositionHcapF}.
Thus $Z^*\cap S=\langle Z\rangle\cap S$.
This implies that $tv\in Z^*$. Since $tv\in A^*t$, we have
$f(r)uf(s)v=f(r)qf(s)$ and thus $uf(s)v=qf(s)$ for some $q\in S$. 
Since $Z^*$ is right unitary, $f(r),f(r)uf(s)v\in Z^*$ imply $uf(s)v=qf(s)\in
Z^*$.
In turn, since $Z^*$ is left unitary, $qf(s),f(s)\in Z^*$ imply $q\in
Z^*$
and thus $q\in Z^*\cap S$.
Let $w\in T$ be such that $f(w)=q$. Then $rws$ is in $T$. This
shows that $T$ is recurrent.
\end{proof}

We prove a series of lemmas. In each of them, we consider 
a uniformly recurrent tree set
$S$ containing $A$
and a coding morphism $f:B^*\rightarrow A^*$ for an $S$-maximal
bifix code $Z$. We set $T=f^{-1}(S)$. We choose $w\in T$ and set $v=f(w)$.
Let also
$Y=\RR_T(w)$. Then $Y$ is a $w^{-1}T$-maximal prefix code.
 Let $X=f(Y)$ or equivalently
$X=Y\circ_f Z$. Then, since $f(w^{-1}T)=v^{-1}S$, by
Proposition~\ref{propositionMaxPref} (i), $X$ is a $v^{-1}S$-maximal
prefix code.

Finally we set $U=\RR_S(v)$. Let $\alpha:C^*\rightarrow A^*$
be a coding morphism for $U$.
Since $X\subset\Gamma_S(v)$, we have $X\subset U^*$. Since $uU^*\cap
X\ne\emptyset$
for any $u\in U$, we have $\alp_U(X)=U$. Thus, by
Proposition~\ref{prop266},
we have $X=W\circ_\alpha U$ where $W$ is
the prefix code such that $\alpha(W)=X$.

\begin{lemma}\label{lemma3}
We have $X^*\cap v^{-1}S=U^*\cap Z^*\cap v^{-1}S$. 
\end{lemma}
\begin{proof}
Indeed, the left handside
is clearly included in the right one. Conversely, consider $x\in
U^*\cap Z^*\cap v^{-1}S$. Since $x\in U^*\cap v^{-1}S$, $\alpha^{-1}(x)$ is in
$\alpha^{-1}(v^{-1}S)=
\alpha^{-1}(\Gamma_S(v))\cup \{1\}$ by Proposition~\ref{propositionRecurrent}.
Thus $x\in \Gamma_S(v)\cup \{1\}$. Since $x\in Z^*$, 
$f^{-1}(x)\in \Gamma_T(w)\cup \{1\}\subset Y^*$.
Therefore $x$ is in $f(Y^*)=X^*$.
\end{proof}
We set for simplicity $d=d_Z(S)$.  Set $H=\alpha^{-1}(v^{-1}S)$. By
Theorem~\ref{propositionReturns},
$H$ is a uniformly recurrent tree set.
\begin{lemma}\label{lemma4}
The set $W$ is a finite $H$-maximal bifix code and 
$d_W(H)=d$. 
\end{lemma}
\begin{proof}
Since $X$ is a prefix code, $W$ is a prefix code. Since $X$
is $v^{-1}S$-maximal, $W$
is $\alpha^{-1}(v^{-1}S)$-maximal by
Proposition~\ref{propositionMaxPref} (ii)
and thus $H$-maximal since $H=\alpha^{-1}(v^{-1}S)$.

Let $x,y\in C^*$ be such that $xy,y\in W$.
Then $\alpha(xy),\alpha(y)\in X$ imply $\alpha(x)\in Z^*$.
Since on the other hand, $\alpha(x)\in U^*\cap v^{-1}S$,
we obtain by Lemma~\ref{lemma3} that $\alpha(x)\in X^*$.
This implies $x\in W^*$ and thus $x=1$ since $W$ is a prefix code.
This shows that $W$ is a suffix code.

To show that $d_W(H)=d$, we consider the morphism $\varphi$
from $A^*$ onto the group $G$ which is the representation of 
$F_A$ on the right cosets of $\langle Z\rangle$.
Set $J=\varphi(Z^*)$. Thus $J$ is a subgroup of index $d$
of $G$.
By Theorem~\ref{theoremJulien}, the set $U$ is a basis of the free
group on $A$. Therefore, since $G$ is a finite group,
 the restriction of $\varphi$
to $U^*$ is surjective.
Set $\psi=\varphi\circ\alpha$. Then $\psi:C^*\rightarrow G$
is a morphism which is onto since $U=\alpha(C)$ generates
the free group on $A$. Let $V$ be the group code of degree $d$ such
that $V^*=\psi^{-1}(J)$. Then $W=V\cap H$, as we will show now. 

Indeed, set $W'=V\cap H$.
If $t\in W$, then $\alpha(t)\in X$ and thus $\alpha(t)\in Z^*$.
Therefore $\psi(t)\in J$ and $t\in V^*$. This shows that $W\subset
W'^*$.
Conversely, if $t\in W'$, then $\psi(t)\in J$ and thus $\alpha(t)\in
Z^*$.
Since on the other hand $\alpha(t)\in U^*\cap S$, we obtain
$\alpha(t)\in X^*$ by Lemma \ref{lemma3}. This implies $t\in W^*$
and shows that $W'\subset W^*$. 

 Thus, since $H$ is a uniformly recurrent tree set, by
Theorem~\ref{theoremGroupCode}, $W$ is a basis of a subgroup
of index $d$. Thus
  $d_W(H)=d$ by Theorem~\ref{theoremFIB}.
\end{proof}
\begin{lemma}\label{lemma5}
The set $Y$ is finite.
\end{lemma}
\begin{proof}
Since $W$ and $U$ are finite, the set $X=W\circ U$ is finite. Thus
$Y=f^{-1}(X)$ is finite.
\end{proof}

\begin{proofof}{of Theorem~\ref{theoremNormal}}
Let $S$ be a uniformly recurrent tree set containing $A$
and let $f:B^*\rightarrow A^*$ be a coding morphism for a finite
$S$-maximal bifix code $Z$.
Set $T=f^{-1}(S)$. 

By Lemma~\ref{lemma1}, $T$ is recurrent. By Lemma~\ref{lemma5} any set
of first return words $Y=\RR_T(w)$ is finite. Thus, by Proposition~\ref{propReturnsFinite}, $T$ is uniformly recurrent.
By Theorem~\ref{InverseImageTree}, $T$ is a tree set.

Thus we conclude that $T$ is a uniformly recurrent tree set.
\end{proofof}
Note that since $T$ is a uniformly recurrent tree set, the set $Y$ is not
only finite as asserted in Lemma~\ref{lemma5} but is in fact a basis
of the free group on $B$, by Theorem~\ref{theoremJulien}.

We illustrate the proof with the following example.
\begin{example}
Let $S$ be the Fibonacci set on $A=\{a,b\}$ and let $Z=S\cap A^2=\{aa,ab,ba\}$. Thus
$Z$ is an $S$-maximal bifix code of $S$-degree $2$.
Let $B=\{c,d,e\}$ and let $f:B^*\rightarrow A^*$ be the coding
morphism defined by $f(c)=aa$, $f(d)=ab$ and $f(e)=ba$.
Part of the set $T=f^{-1}(S)$ is represented in
Figure~\ref{figureSetK}
on the left (this set is the same as the set of Example~\ref{exampleDivision}
with $a,b,c$ replaced by $c,d,e$).

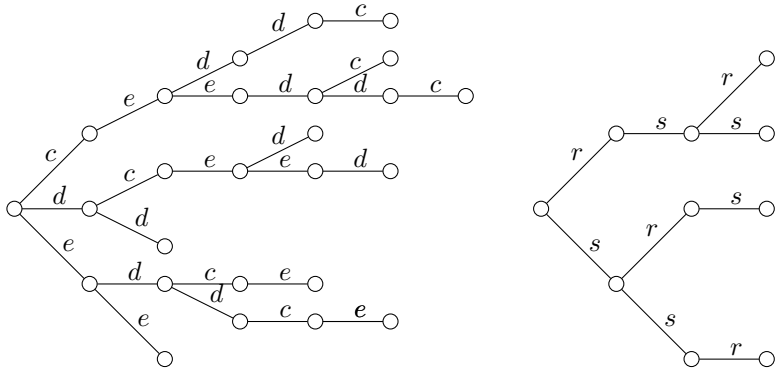
\begin{figure}[hbt]
\centering
\gasset{Nadjust=wh,AHnb=0,ELpos=60,ELdist=.6}
\begin{picture}(120,45)(0,10)
\put(0,0){
\begin{picture}(60,45)
\node(1)(0,30){}\node(c)(10,40){}\node(d)(10,30){}\node(e)(10,20){}

\node(ce)(20,45){}
\node(dc)(20,35){}\node(dd)(20,25){}
\node(ed)(20,20){}\node(ee)(20,10){}

\node(ced)(30,50){}\node(cee)(30,45){}
\node(dce)(30,35){}
\node(edc)(30,20){}\node(edd)(30,15){}

\node(cedd)(40,55){}\node(ceed)(40,45){}
\node(dced)(40,40){}\node(dcee)(40,35){}
\node(edce)(40,20){}\node(eddc)(40,15){}

\node(ceddc)(50,55){}\node(ceedc)(50,50){}
\node(ceedd)(50,45){}\node(dceed)(50,35){}
\node(eddce)(50,15){}

\node(ceeddc)(60,45){}

\drawedge(1,c){$c$}\drawedge(1,d){$d$}\drawedge(1,e){$e$}
\drawedge(c,ce){$e$}\drawedge(d,dc){$c$}\drawedge(d,dd){$d$}
\drawedge(e,ed){$d$}\drawedge(e,ee){$e$}
\drawedge(ce,ced){$d$}\drawedge(ce,cee){$e$}
\drawedge(dc,dce){$e$}
\drawedge(ed,edc){$c$}\drawedge(ed,edd){$d$}
\drawedge(ced,cedd){$d$}\drawedge(cee,ceed){$d$}
\drawedge(dce,dced){$d$}\drawedge(dce,dcee){$e$}
\drawedge(edc,edce){$e$}
\drawedge(edd,eddc){$c$}\drawedge(eddc,eddce){$e$}

\drawedge(cedd,ceddc){$c$}\drawedge(ceed,ceedc){$c$}
\drawedge(ceed,ceedd){$d$}\drawedge(dcee,dceed){$d$}
\drawedge(eddc,eddce){$e$}

\drawedge(ceedd,ceeddc){$c$}
\end{picture}
}
\put(70,0){
\begin{picture}(60,45)
\node(1)(0,30){}\node(r)(10,40){}\node(s)(10,20){}
\node(rs)(20,40){}\node(sr)(20,30){}\node(ss)(20,10){}
\node(rsr)(30,50){}\node(rss)(30,40){}
\node(srs)(30,30){}\node(ssr)(30,10){}

\drawedge(1,r){$r$}\drawedge(1,s){$s$}
\drawedge(r,rs){$s$}\drawedge(s,sr){$r$}\drawedge(s,ss){$s$}
\drawedge(rs,rsr){$r$}\drawedge(rs,rss){$s$}
\drawedge(sr,srs){$s$}\drawedge(ss,ssr){$r$}
\end{picture}
}
\end{picture}
\caption{The sets $T$ and $H$.}\label{figureSetK}
\end{figure}
The sets $Y=\RR_T(c)$ and $X=f(Y)$ are 
\begin{displaymath}
Y=\{eddc,eedc,eeddc\},\quad X=\{baababaa,babaabaa,babaababaa\}.
\end{displaymath}
On the other hand, the set $U=\RR_S(aa)$ is $U=\{baa,babaa\}$.
Let $C=\{r,s\}$ and let $\alpha:C^*\rightarrow A^*$ be the
coding morphism for $U$ defined by $\alpha(r)=baa$,
$\alpha(s)=babaa$. Part of the set $H=\alpha^{-1}((aa)^{-1}S)$
is represented in Figure~\ref{figureSetK} on the right.
Then we have $W=\{rs,sr,ss\}$ which is an $H$-maximal bifix code
of $H$-degree $2$ in agreement with Lemma~\ref{lemma4}.
\end{example}

The following example shows that the condition that $S$ is a tree
set is necessary.
\begin{example}\label{example511}
Let $S$ be the set of factors of $(ab)^*$. The set $S$
does not satisfy the tree condition since $G(\epsilon)$ is not
connected.
Let $X=\{ab,ba\}$.
The set $X$ is a finite $S$-maximal bifix code. Let
$f:\{u,v\}^*\rightarrow A^*$ be the coding morphism for $X$ defined by
$f(u)=ab$, $f(v)=ba$. Then $f^{-1}(S)=u^*\cup v^*$ is not recurrent.
\end{example}
\subsection{Composition of bifix codes}\label{sectionComposition}
In this section, we use Theorem~\ref{theoremNormal} to prove a result
showing that in a uniformly recurrent tree set, the degrees of the terms of a composition
of maximal bifix codes are multiplicative 
(Theorem~\ref{theoremCompositionBifix}).

 The following result
 is proved in~\cite[Proposition 11.1.2]{BerstelPerrinReutenauer2009} for a more
general class of codes (including all finite codes and not only
finite bifix codes), but in the case of $S=A^*$.
\begin{theorem}\label{theoremCompositionBifix}
Let $S$ be a uniformly recurrent tree set and let $X,Z\subset S$ be finite bifix codes such
that $X$ decomposes into $X=Y\circ_f Z$ where $f$ is a coding morphism for $Z$.
 Set
$T=f^{-1}(S)$. Then $X$ is an $S$-maximal bifix code 
if and only if  $Y$ is a $T$-maximal bifix code and $Z$ is an
$S$-maximal bifix code.
Moreover, in this case
\begin{equation}
d_X(S)=d_Y(T)d_Z(S). \label{eqDegreesMult}
\end{equation}
\end{theorem}
\begin{proof}
Assume first that $X$ is an $S$-maximal bifix code.
By Proposition~\ref{propositionMaxPref} (ii),  $Y$ is a
$T$-maximal prefix code and 
$Z$ is an $S$-maximal prefix code. This implies that  $Y$ is
a $T$-maximal bifix code and that $Z$ is an $S$-maximal
bifix code.

The converse also holds by Proposition~\ref{propositionMaxPref}.

To show Formula~\eqref{eqDegreesMult},
let us first observe that there exist words $w\in S$ such that for every
parse $(v,x,u)$ of $w$ with respect to $X$, the word $x$ is not a factor of $X$.
Indeed, let $n$ be the maximal length of the words of $X$.
Assume that the length of $w\in S$ is larger than
$3n$ . Then
if $(v,x,u)$ is a parse of $w$, we have $|u|,|v|<n$ and thus $|x|>n$.
This implies that $x$ is not a factor of $X$.

Next, we observe that by Theorem~\ref{theoremNormal}, the set
$T$ is a uniformly recurrent tree set and thus in particular, it is recurrent.

Let $w\in S$ be a word  with the above property. Let $\Pi_X(w)$ denote the set of parses of $w$ with
respect to $X$ and $\Pi_Z(w)$ the set of its parses with respect to
$Z$. We define a map $\varphi:\Pi_X(w)\rightarrow \Pi_Z(w)$
as follows. Let $\pi=(v,x,u)\in \Pi_X(w)$. Since $Z$ is a bifix code,
there is a unique way to
write $v=sy$ and $u=zr$
with $s\in A^*\setminus A^*Z$, $y,z\in Z^*$ and $r\in A^*\setminus
ZA^*$.
We set $\varphi(\pi)=(s,yxz,r)$. The triples $(y,x,z)$ are in
bijection
with the parses of $f^{-1}(yxz)$ with respect to $Y$. Since
$x$ is not a factor of $X$ by the hypothesis made on $w$, and since
$T$ is recurrent, there
are $d_Y(T)$ such triples. This shows Formula~\eqref{eqDegreesMult}.
\end{proof}

\begin{example}\label{exampleCodeGiuseppina}
Let $S$ be the Fibonacci set.
Let $B=\{u,v,w\}$ and $A=\{a,b\}$. Let $f:B^*\rightarrow A^*$
be the morphism defined by $f(u)=a$, $f(v)=baab$ and $f(w)=bab$.
Set $T=f^{-1}(S)$. The words of length at most $3$ of $T$ are
represented
on Figure~\ref{figureG}. 
\begin{figure}[hbt]
\centering
\gasset{AHnb=0,Nadjust=wh}
\begin{picture}(50,35)
\node(1)(0,10){}
\node(u)(15,25){}\node(v)(15,10){}\node(w)(15,0){}

\node(uu)(30,30){}\node(uv)(30,25){}\node(uw)(30,18){}
\node(vu)(30,10){}
\node(wu)(30,0){}

\node(uuv)(45,35){}\node(uuw)(45,30){}
\node(uvu)(45,25){}
\node(uwu)(45,18){}
\node(vuu)(45,12){}\node(vuv)(45,5){}
\node(wuu)(45,0){}

\drawedge(1,u){$u$}\drawedge(1,v){$v$}\drawedge(1,w){$w$}
\drawedge[ELpos=60](u,uu){$u$}\drawedge[ELpos=60](u,uv){$v$}\drawedge[ELpos=60](u,uw){$w$}
\drawedge(v,vu){$u$}
\drawedge(w,wu){$u$}

\drawedge(uu,uuv){$v$}\drawedge(uu,uuw){$w$}
\drawedge(uv,uvu){$u$}
\drawedge(uw,uwu){$u$}
\drawedge(vu,vuu){$u$}\drawedge(vu,vuv){$v$}

\drawedge(wu,wuu){$u$}
\end{picture}
\caption{The words of length at most $3$ in $T$.}\label{figureG}
\end{figure}
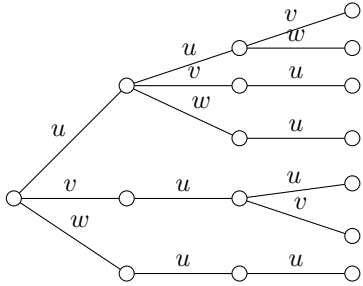

The set $Z=f(B)$ is an $S$-maximal bifix code of $S$-degree $2$
(it is the unique $S$-maximal bifix code of $S$-degree $2$
with kernel $\{a\}$).
Let $Y=\{uu,uvu,uw,v,wu\}$, which is
a $T$-maximal bifix code of $T$-degree $2$
(it is the unique $T$-maximal bifix code of $T$-degree $2$
with kernel $\{v\}$).

The code $X=f(Y)$ is the
  $S$-maximal bifix code of $S$-degree $4$ shown on
  Figure~\ref{figureCodeGiuseppina}. 
\begin{figure}[hbt]
\centering
\gasset{AHnb=0,Nadjust=wh}
\begin{picture}(60,35)(0,3)
\node(1)(0,20){}
\node(a)(10,30){}\node(b)(10,10){}
\node[Nmr=0](aa)(20,35){}\node(ab)(20,25){}\node(ba)(20,10){}
\node(aba)(30,25){}\node(baa)(30,15){}\node(bab)(30,5){}
\node(abaa)(40,30){}\node[Nmr=0](abab)(40,20){}
\node[Nmr=0](baab)(40,15){}\node[Nmr=0](baba)(40,5){}
\node(abaab)(50,30){}\node[Nmr=0](abaaba)(60,30){}
\drawedge(1,a){$a$}\drawedge[ELside=r](1,b){$b$}
\drawedge(a,aa){$a$}\drawedge(a,ab){$b$}\drawedge(b,ba){$a$}
\drawedge(ab,aba){$a$}\drawedge(ba,baa){$a$}\drawedge[ELside=r](ba,bab){$b$}
\drawedge(aba,abaa){$a$}\drawedge[ELside=r](aba,abab){$b$}
\drawedge[ELside=r](baa,baab){$b$}\drawedge(bab,baba){$a$}
\drawedge(abaa,abaab){$b$}\drawedge(abaab,abaaba){$a$}
\end{picture}
\caption{An $S$-maximal bifix code of $S$-degree 4.}\label{figureCodeGiuseppina}
\end{figure}
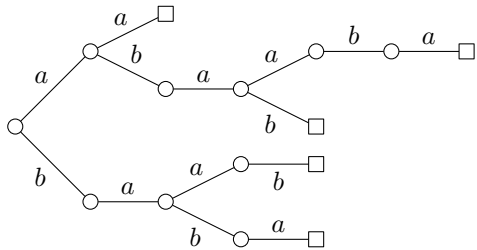

\end{example}

Example~\ref{exampleNotMult} shows that Formula~\eqref{eqDegreesMult} does
not hold if $S$ is not a tree set. 

\begin{example}\label{exampleNotMult}
Let $S=F(ab)^*$ (see Example~\ref{example511}). Let $Z=\{ab,ba\}$ and let $X=\{abab,ba\}$. We have
$X=Y\circ_f Z$ for $B=\{u,v\}$, $f:B^*\rightarrow A^*$ defined
by $f(u)=ab$ and $f(v)=ba$ with $Y=\{uu,v\}$. The codes $X$ and $Z$
are  $S$-maximal bifix codes and $d_Z(S)=2$. We have $d_X(S)=3$
since $abab$ has three parses. Thus $d_Z(S)$ does not
divide $d_X(S)$.
\end{example}
\bibliographystyle{plain}
\bibliography{bifixDecoding}
\end{document}